\documentclass[11pt]{amsart}  
\usepackage{amsmath,amssymb,amsthm,amsfonts}
\usepackage{mathrsfs}
\usepackage{xcolor}
\usepackage[all,cmtip]{xy}

\usepackage{geometry}
\newtheorem{theorem}{Theorem}[section]
\newtheorem{proposition}[theorem]{Proposition}

\newtheorem{lemma}[theorem]{Lemma}

\newtheorem*{question*}{Question}

\theoremstyle{definition}
\newtheorem{definition}[theorem]{Definition}

\newtheorem*{definition*}{Definition}

\theoremstyle{remark}
\newtheorem{remark}[theorem]{Remark}

\newcommand{\eps}{\varepsilon}

\newcommand{\EE}{\mathbb{E}}
\newcommand{\CC}{\mathbb{C}}

\newcommand{\RR}{\mathbb{R}}

\newcommand{\NN}{\mathbb{N}}
\newcommand{\FF}{\mathbb{F}}
\newcommand{\TT}{\mathbb{T}}
\newcommand{\ZZ}{\mathbb{Z}}

\newcommand{\cH}{\mathcal{H}}

\newcommand{\cC}{\mathcal{C}}

\newcommand{\cL}{\mathcal{L}}

\def\dH{\dim_{\mathcal{H}}}

\title{Polynomial Szemerédi for sets with large Hausdorff dimension on the Torus}

\author{Guo-Dong Hong}
\address{Department of Mathematics, California Institute of Technology, Pasadena, CA 91125, USA}
\email{ghong@caltech.edu}

\begin{document}
\maketitle

\begin{abstract}
Let $\mathcal{P}= \{P_1, \cdots, P_{k}\in \RR[y]\}$ be a collection of polynomials with distinct degrees and zero constant terms. We proved that there exists $\epsilon=\epsilon(\mathcal{P})>0$ such that, for any compact set $E \subset \TT$ with $\dH(E)>1-\epsilon$, we can find $y\neq 0$ so that $\{x,x+P_1(y), \cdots,x+P_k(y)\} \subset E$. The proof relies on a suitable version of the Sobolev smoothing inequality with ideas adapted from Peluse \cite{P19}, Durcik and Roos \cite{DR24}, and Krause, Mirek, Peluse, and Wright \cite{KMPW24}. As a byproduct of our Sobolev smoothing inequality, we demonstrated that the divergence set of the pointwise convergence problem for certain polynomial multiple ergodic averages has Hausdorff dimension strictly less than one.
\end{abstract}

\section{Introduction}

\subsection{Polynomial Szemerédi theorem}
A central topic in Ramsey theory is to find patterns for objects of large size. One particular instance that attracts considerable interest is the celebrated Szemerédi theorem \cite{S75}, which states that any subset of $\ZZ$ with positive upper density must contain arbitrarily long arithmetic progressions. Later on, Bergelson and Leibman \cite{BL96} generalized Szemerédi's theorem to more general polynomial progressions. To be more precise, let $\mathcal{P}= \{P_1, \cdots, P_{k} \in \ZZ[y]\}$ be a collection of polynomials with zero constant terms. Bergelson and Leibman proved that any subset of $\NN$ with positive upper density must contain the following progressions:
\[
x,x+P_1(y), \cdots,x+P_k(y)
\]
for some $y \neq 0$.

The purpose of this paper is to address a related problem in the Euclidean setting. In this context, we say that a subset of $\RR$ has a large size if it has a positive Lebesgue measure. It is well known that, by the Lebesgue density theorem, any subset of $\RR$ with positive Lebesgue measure must contain arbitrarily long polynomial progressions with zero constant terms, which can be regarded as a continuous version of Bergelson and Leibman's theorem. A natural follow-up question is whether a similar result can be proven for sparser sets, that is, measure-zero sets. In this circumstance, we say that a measure-zero set has a large size if it has a large Hausdorff dimension. However, a construction due to Keleti \cite{K99} shows that there exists a measure-zero subset of $\RR$ with Hausdorff dimension one that does not contain any three-term arithmetic progression. Therefore, extending Bergelson and Leibman's theorem in this context is generally impossible.

In light of Keleti's construction, Laba and Pramanik \cite{LP09} proposed a Fourier decay condition, which is motivated by the pseudorandomness notion in the discrete setting, together with the large Hausdorff dimension condition on the subset of $\RR$ to ensure the existence of three-term arithmetic progressions. Afterwards, Fraser, Guo, and Pramanik \cite{FGP22} proved the existence of the following nontrivial three-term polynomial progressions
\[x,x+y,x+P(y),\]
where $P\in \RR[y]$ is a polynomial with degree at least two with no constant terms, under similar conditions to those in Laba and Pramanik. And this result was extended to the three-term polynomial progressions 
\[x,x+P_1(y),x+P_2(y),\]
where $P_1, P_2\in \RR[y]$ are linearly independent polynomials with zero constant terms by Chen, Guo, and Li \cite{CGL21} later. However, it is unclear in these papers whether thisFourier decay condition is necessary to ensure the existence of the nonlinear patterns.

Surprisingly, Kuca-Orponen-Sahlsten \cite{KOS23} showed that having a large Hausdorff dimension is the only necessary assumption to guarantee the existence of certain nonlinear patterns. More specifically, they proved that any compact subset $E \subset \RR^2$ with large Hausdorff dimension must contain
\[\underline{x}, \underline{x}+(y,y^2)\]
for some $y \neq 0$. Their results demonstrated the distinction between the linear patterns and the nonlinear patterns. By adapting their method, Bruce-Pramanik \cite{BP23} studied more general two-term nonlinear patterns for subsets in higher dimensions, and Zhu \cite{z24} proved the existence of the quadratic Roth configuration
\[
x,x+y,x+y^2
\]
for subsets of $\RR$ only under having a large Hausdorff dimension. See also Krause \cite{K19} for related results. 
It remains an interesting question to determine whether one can extend their results to identify longer nonlinear patterns in this context, which may be a challenging problem.

The first main result in this paper is the first step toward answering the question above.

\begin{theorem} \label{Thm_Progression}
    Let $\mathcal{P}= \{P_1, \cdots, P_{k}\in \RR[y]\}$ be a collection of polynomials with distinct degrees and zero constant terms. Then there exists $\epsilon=\epsilon(\mathcal{P})>0$ such that the following holds. Let $E \subset \TT$ be any compact set with Hausdorff dimension $\dH(E)>1-\epsilon$. Then there exists $y\neq 0$ so that $\{x,x+P_1(y), \cdots,x+P_k(y)\} \subset E$.
\end{theorem}

Another way to rephrase Theorem \ref{Thm_Progression} is that any periodic subset of $\RR$ with large Hausdorff dimension must contain certain nontrivial polynomial progressions. In this sense, this result is a simpler problem since we have an additional periodic assumption. According to our proof, the common difference $y$ is comparatively large, which contrasts with all previous results in the literature. The main distinction is that we take advantage of the averaging effect in this periodic setting, which has a stronger consequence, as we will discuss in the next subsection.

\subsection{Sobolev smoothing inequality}
To show the existence of the polynomial progression, it is common to study the following counting operator:
\[
\Lambda_{\mathcal{P};N}(f_0,...,f_k):=  \frac{1}{N} \int_0^N \int_{\TT}f_0(x) \prod_{i=1}^k f_i( x+ P_i(y)) \, dx \,dy,
\]
where $\mathcal{P}= \{P_1, \cdots, P_{k}\in \RR[y]\}$ is a collection of polynomials. The most essential part of this paper is the following version of the Sobolev smoothing inequality.

\begin{theorem} \label{Thm_Sobolev}
    Let $k \in \NN$, $N \geq 1$, and let $\mathcal{P}= \{P_1, \cdots, P_{k}\in \RR[y]\}$ be a collection of polynomials with distinct degrees and zero constant terms.
    Then there exist $\sigma=\sigma(\mathcal{P})>0$ and $C=C(\mathcal{P}), c=c(\mathcal{P})>0$ such that
        \[
        \left|\Lambda_{\mathcal{P};N}(f_0,...,f_k) -\prod_{i=0}^k \int_{\TT} f_i \right| \lesssim_{\mathcal{P}}
        N^{-C} \cdot \min_{i=0,\cdots, k}  \|f_i\|^{c}_{\cH^{-\sigma}(\TT)} ,
        \]
        for any 1-bounded measurable functions $f_0, \cdots, f_k $.
\end{theorem}
The Sobolev smoothing inequality was introduced by Bourgain \cite{B88} in a different context, where he utilized it to obtain a quantitative count of the quadratic Roth configuration
\[
x,x+y,x+y^2.
\]
for subsets of $\RR$ with positive Lebesgue measure.
Subsequently, the Sobolev smoothing inequality corresponding to more general three-term polynomial progressions was extended by Durcik, Guo, and Roos \cite{DGR19} for 
\[x,x+y,x+P(y),\]
where $P\in \RR[y]$ is a polynomial with degree at least two with no constant terms, and by Chen, Guo, and Li \cite{CGL21} for
\[x,x+P_1(y),x+P_2(y),\]
where $P_1, P_2\in \RR[y]$ are polynomials with distinct degrees and zero constant terms. Until recently, Krause, Mirek, Peluse, and Wright \cite{KMPW24} significantly generalized the Sobolev smoothing inequality corresponding to longer-term polynomial progressions
\[
x,x+P_1(y), \cdots,x+P_k(y),
\]
where $\mathcal{P}= \{P_1, \cdots, P_{k}\in \RR[y]\}$ is a collection of polynomials with distinct degrees and zero constant terms. We can regard Theorem \ref{Thm_Sobolev} as a strengthened version of the Sobolev smoothing inequality obtained in \cite{KMPW24} under the extra periodic assumption.

The work in \cite{KMPW24} was inspired by recent advances of the polynomial Szemerédi theorem in the finite field setting. Bourgain and Chang \cite{BC17} initiated the study of the nonlinear Roth theorem in the finite field setting, and they proved the asymptotic formulas of the quantitative count for the following three-term nonlinear progressions
\[
x,x+y,x+y^2 \quad \text{and} \quad x,x+y,x+1/y.
\]
Their results were later independently extended to more general three-term polynomial progressions by Dong, Li, and Sawin \cite{DLS20}, and by Peluse \cite{P18}; and more general three-term rational function progressions by the author and Lim \cite{HL25}. In \cite{P19}, Peluse made a breakthrough by introducing a new method, now known as the degree-lowering method, to deal with longer-term polynomial progressions.
Consider the counting operators for $\mathcal{P}= \{P_1, \cdots, P_{k}\}$ in the finite field setting as follows
\[\Lambda_{\mathcal{P}}( f_0,...,f_{k})=
    \EE_{x,y \in \FF_p}f_0(x) \prod_{i=1}^k f_i( x+ P_i(y)),\]
where $\EE_{x \in \FF_p}$ denotes $\frac{1}{p} \sum_{x \in \FF_p}$. Peluse \cite{P19} proved that given any collection of linearly independent polynomials with zero constant terms, denoted by $\mathcal{P}= \{P_1, \cdots, P_{k}\in \ZZ[t]\}$, there exists $C=C(\mathcal{P})>0$ such that
\begin{equation} \label{sarah's theorem}
    \Lambda_{\mathcal{P}}( f_0,...,f_{k})= \prod_{i=0}^k \EE f_i +O_{\mathcal{P}}(p^{-c}),
\end{equation}
for any 1-bounded functions $f_0, \cdots, f_k $. Theorem \ref{Thm_Sobolev} and the Sobolev smoothing inequality obtained in \cite{KMPW24} can be viewed as a generalization of Peluse's theorem \eqref{sarah's theorem} in the continuous setting. We also refer to Durcik and Roos \cite{DR24}, which provided a new proof of Bourgain's Sobolev smoothing inequality, adapting the techniques from Peluse \cite{P19}. Ideas in \cite{P19}, \cite{KMPW24}, and \cite{DR24} play important parts in the proof of Theorem \ref{Thm_Sobolev}.

When we choose $f_i$ to be the characteristic function of a measurable set $E \subset\TT$, Theorem \ref{Thm_Sobolev} immediately implies the quantitative count of the polynomial progressions in $E$, when $N$ tends to infinity: this averaging effect is the main distinction with the previous setting in $\RR$, as in Bourgain \cite{B88}, so that we can deduce the existence of a nontrivial polynomial progression in a comparatively more straightforward manner.
However, to study the problem of the fractal set, we need to remove the 1-boundedness assumption. In general, we do not know if Theorem \ref{Thm_Sobolev} holds for functions in the negative Sobolev space. But we can find a suitable version of the Sobolev smoothing inequality, which is Theorem \ref{Thm_Counting} below, that has implications for Theorem \ref{Thm_Progression}.
We believe that formulating a proper Sobolev smoothing inequality in this setting and overcoming the difficulties associated with this generalization are the main novelties of this paper.

\subsection{Divergence set of the pointwise convergence problem}
Another perspective to look at the polynomial Szemerédi theorem is through examining the following multiple ergodic averages:
\begin{equation} \label{multiple ergodic average}
    \mathcal{A}_{\mathcal{P};N}(f_1,...,f_k)(x):= 
\frac{1}{N} \int_0^N \prod_{i=1}^{k} f_i(x+P_i(y)) \,dy.
\end{equation}
where $\mathcal{P}= \{P_1, \cdots, P_{k}\in \RR[y]\}$ is a collection of polynomials.
 The final result in this paper addresses the pointwise convergence problem for \eqref{multiple ergodic average}. It turns out that, in the current setting, we can prove a stronger pointwise convergence result.

\begin{theorem} \label{Thm_Convergence}
Let $k \in \NN$ and $\mathcal{P}= \{P_1, \cdots, P_{k}\in \RR[y]\}$ be a collection of polynomials with distinct degrees and zero constant terms.
    Then there exists $\epsilon=\epsilon(\mathcal{P})>0$ such that
    \[
    \dH \{x \in \TT: \lim_{N \to\infty}\mathcal{A}_{\mathcal{P};N}(f_1,...,f_k)(x) \neq \prod_{i=1}^k \int_{\TT} f_i\} \leq 1-\epsilon,
    \]
    for any $f_1, \cdots, f_k \in \cL^{\infty}(\TT)$.
\end{theorem}

There is a long history of studying the convergence of \eqref{multiple ergodic average} for both norm convergence and the pointwise convergence problem when the discrete sum average replaces the integral average in a much more general setting. The $\cL^2$-norm convergence of \eqref{multiple ergodic average} with the discrete sum average to the product of integrals was studied by Frantzikinakis and Kra in \cite{FK05}. See also the recent breakthrough paper by Frantzikinakis and Kuca in \cite{FK25} on the norm convergence problem. 
Moreover, Kosz, Mirek, Peluse, and Wright \cite{KMPW25} settled the pointwise almost everywhere convergence of \eqref{multiple ergodic average} with the discrete sum average shortly after. Both papers \cite{FK25} and \cite{KMPW25} in fact proved stronger multidimensional results in a much more general setting; we refer readers to these papers for more detailed literature and the current state of the art.

In another direction, there are only a few results that deal with the continuous average, as in \eqref{multiple ergodic average}. For example, the pointwise almost everywhere convergence of \eqref{multiple ergodic average} for $\mathcal{P}= \{y, y^2\}$ is covered in Christ, Durcik, Kovač, and Roos \cite{CDKR22}, and Frantzikinakis \cite{F23} proved pointwise almost everywhere convergence of multiple ergodic averages with iterates given by a special collection of Hardy field sequences, which covered some cases we consider here. Results on both papers \cite{CDKR22} and \cite{F23} are also proved in a much more general setting.

To the author's knowledge, Theorem \ref{Thm_Convergence} is the first result in the literature to address the pointwise convergence problem beyond the almost everywhere issue, as we utilize the topology of $\mathbb{R}$, which is absent in general probability spaces. It is an interesting question to ask how small the divergence set could be in Theorem \ref{Thm_Convergence}, and we intend to address this issue in future work.

\subsection{Organization of the paper}
In Section 2, we establish some frequently used notation and provide a brief introduction to the necessary background for the Gowers uniformity norm, standard oscillatory integral estimates, and geometric measure theory. In Section 3, we give the proof of the Sobolev smoothing inequality, Theorem \ref{Thm_Sobolev}. In Section 4, we upgrade Theorem \ref{Thm_Sobolev} to Theorem \ref{Thm_Counting}, and we apply it to prove Theorem \ref{Thm_Progression}. Finally, in Section 5, we prove the divergence set of the pointwise convergence problem, Theorem \ref{Thm_Convergence}, by using Theorem \ref{Thm_Counting}.

\subsection{Acknowledgements}
The author would like to thank Zi Li Lim and Polona Durcik for comments
and his advisor, David Conlon, for his
encouragement and support.
The author is supported by the MOE Taiwan-Caltech Fellowship during the conduct of this research.

\section{Notation and preliminaries}
In this section, we present the basic notation and necessary background that will be used throughout the paper.
\subsection{Notation and conventions}
Throughout the paper, all functions we consider are measurable. We write $A\lesssim B$ to mean that there exists an absolute constant $C>0$ such that $A \leq C B$. We write $A \sim B$ if $A\lesssim B \lesssim A$. If the constant depends on other parameter, say $C=C(\mathcal{P})$, then we write  $A \lesssim_{\mathcal{P}} B$ to mean that $A \leq C(\mathcal{P}) B$. We also use the Landau notation. We write $A=O(B)$ if $|A|\lesssim B$, and $A=O_{\mathcal{P}}(B)$ if $|A|\lesssim_{\mathcal{P}} B$. Next, let $1_S$ denote the indicator function of a set $S$, which means $1_S(x)=1$, if $x \in S$; and $1_S(x)=0$, if $x \notin S$. We use $\NN$ to denote the set of positive integers.

Let $\TT \cong \RR/\ZZ$ denote the Torus. Given $x \in \TT$ and $ \xi \in \ZZ$, we set
\[
e(x):=e^{2 \pi i x}.
\]
For any function $f: \TT \rightarrow \CC$, denote the $\cL^p$-norm, for $p \in (1,\infty)$, as 
\[
\|f\|_{\cL^p(\TT)} := \left( \int_{\TT} |f(x)|^p \,dx \right)^{1/p}.
\]
And for $p= \infty$, denote the $\cL^{\infty}$-norm as 
\[
\|f\|_{\cL^{\infty}(\TT)} := \sup_{x \in \TT} |f(x)|.
\]
We call a function $f \in \cL^{\infty}(\TT)$ 1-bounded if $\|f\|_{\cL^{\infty}(\TT)} \leq 1$.
For any $f \in\cL^1(\TT)$, the Fourier coefficient of $f$ is defined as
\[
\widehat{f}(\xi):= \int_{\TT} e(-\xi x) f(x) \, dx.
\]
Similarly, for any function $f: \ZZ \rightarrow \CC$, denote the $l^p$-norm, for $p \in (1,\infty)$, as
\[
\|f\|_{l^p(\TT)} := \left( \sum_{x \in\ZZ} |f(x)|^p \right)^{1/p}.
\]
And for $p= \infty$, denote the $l^{\infty}$-norm as 
\[
\|f\|_{l^{\infty}(\ZZ)} := \sup_{x \in \ZZ} |f(x)|.
\]
Then Parseval's identity reads 
\[
\|f\|_{\cL^2(\TT)}= \|\widehat{f}\|_{l^2(\ZZ)}.
\]
For $\sigma>0$, and for any function $f: \ZZ \rightarrow \CC$, denote the negative Sobolev norm as
\[
\|f\|_{\cH^{-\sigma}(\TT)} := \left( \sum_{\xi \in \ZZ }|\widehat{f}(\xi)|^2 (1+|\xi|^2)^{-\sigma/2}  \right)^{1/2}.
\]
\subsection{Gowers $U^s$-norm}
Given $x,h \in \TT$, define the multiplicative derivative
\[
\triangle_hf(x):=f(x)\overline{f}(x+h),
\]
for any $f \in \cL^{\infty}(\TT)$. And for $h_1, \cdots, h_s \in \TT$, we define
\[
\triangle_{h_1, \cdots, h_s}f(x):=\triangle_{h_1} \cdots \triangle_{h_s}f(x)= \prod_{\underline{w}\in\{0,1\}^s} \cC^{|\underline{w}|}f(x+\underline{w}\cdot \underline{h}),
\]
where $\cC f:= \overline{f}$ is the conjugation operator, and $\underline{h}:=\{h_1, \cdots, h_s\}$.

Next, for any $s \in \NN$, we define the Gowers $U^s$-norm by
\[
\|f\|_{U^s(\TT)} := \left( \int_{\TT^{s+1}}  \triangle_{h_1, \cdots, h_s}f(x) \,dx \, dh_1,...,\,dh_s \right)^{1/2^s}.
\]
For $s=2$, we have a nice formula for the $U^2$-norm in terms of the Fourier coefficient
\[
\|f\|_{U^2(\TT)} = \|\widehat{f}\|_{l^4(\ZZ)}.
\]
We will frequently use the following fact later
\begin{equation} \label{U^2-inverse theorem}
    \|f\|^4_{U^2(\TT)}=\|\widehat{f}\|^4_{l^4(\ZZ)} \leq \|\widehat{f}\|^2_{l^{\infty}(\ZZ)}
    \leq \left|\int_{\TT}f \right|^2 + \sup_{\xi \in \ZZ-\{0\}} \left|\int_{\TT} e(\xi x)f(x) \, dx \right|^2,
\end{equation}
for any 1-bounded function $f$.
$U^s$-norms are seminorms, and are norms for $s\geq 2$. Also, $U^s$-norms satisfy the following monotonicity property
\begin{equation} \label{Monotonicity}
    \|f\|_{U^s(\TT)} \leq \|f\|_{U^{s+1}(\TT)},
\end{equation}
for any 1-bounded function $f$. Moreover, we have the following Gowers-Cauchy-Schwarz inequality
\begin{equation} \label{Gowers-Cauchy-Schwarz inequality}
\left|
    \int_{\TT^{s+1}} \prod_{\underline{w}\in\{0,1\}^s} \cC^{|\underline{w}|}f_{\underline{w}}(x+\underline{w}\cdot \underline{h}) \,dx \, dh_1,...,\,dh_s \right| \leq \prod_{\underline{w}\in\{0,1\}^s} \|f_{\underline{w}}\|_{U^s(\TT)}.
\end{equation}
We refer the readers to \cite[Chapter 11]{TV06} for more properties of the Gowers norm.
\begin{remark}
    Note that all the notions above can be defined when the underlying space is changed to $\RR$. In this case, we have the notion of the Fourier transform, instead of the Fourier coefficient. See \cite[Section 3, 4]{KMPW24} for more discussion.
\end{remark}

\subsection{Oscillatory integral estimates}
In this subsection, we record the classical van der Corput estimates, which can be found, for example, in \cite[Appendix B]{K22}:
\begin{lemma} \label{Lma_Oscillatory integral}
    Let $P(y)= \sum_{i=1}^d a_i y^i \in \RR[y]$. Then 
    \[
    \left|\int_{0}^1 e(P(y))\, dy\right| \lesssim_d (1+\max_i \, |a_i|)^{-1/d}.
    \]
    In particular, we have
    \[
    \left|\frac{1}{N} \int_0^N e(\xi P(y)) \, dy
    \right|
    \lesssim_d  \left|1+\left(N|\xi|\cdot \max_i \, |a_i|\right)\right|^{-1/d}
    \lesssim_P \left(N|\xi|\right)^{-1/d},
    \]
    for any $\xi \in \ZZ-\{0\}$ and $N \geq1$.
\end{lemma}

\subsection{Frostman measure and the Riesz energy}
This subsection aims to introduce the needed background from geometric measure theory. Most of the results in this subsection can be found in \cite[Chapter 2, 3]{M15}.

We first introduce the notion of the Frostman measure, which plays an important role when studying combinatorial problems in fractal sets:
\begin{definition}[Frostman measure]
        Let $s>0$. We call a measure $\mu$ an s-Frostman measure if 
        \begin{equation} \label{definition of Frostman measure}
            \mu (B(x,r)) \leq  r^s, 
            \forall\, x \in \mathbb{R}^d, \, \forall r>0.
        \end{equation}
\end{definition}
\begin{remark}
    We do not assume an s-Frostman measure $\mu$ to be a probability measure. But it follows from the definition that if $\mu$ has a compact support, then it is a finite measure. In particular, if $\text{spt}(\mu) \subset \TT$, we have $\mu(\TT)\leq 1$.
\end{remark}
The following theorem \ref{Frostman's lemma}, known as Frostman's lemma, says that one can view a fractal set as a Frostman measure in practice.
\begin{theorem}[Frostman's lemma] \label{Frostman's lemma}
        Let $s \in (0,d]$. For a Borel set $E \subset \mathbb{R}^d$, $\mathcal{H}^{s}(E)>0$ iff there exists $\mu \in \mathcal{M}(E)$ such that $\mu$ is an s-Frostman measure, where $\mathcal{H}^{s}$ denotes the $s$-dimensional Hausdorff measure, and $\mathcal{M}(E)$ denotes the set of all finite Borel measures with compact support $\text{spt}(\mu) \subset E$.
\end{theorem}

Next, we introduce the notion of the Riesz energy, which is important from the perspective of harmonic analysis:
\begin{definition}[Riesz energy]
        Let $s>0$. We define the Riesz energy of a Borel measure $\mu$, denoted by $\mathcal{I}_s(\mu)$, as follows:
        \[
            \mathcal{I}_s(\mu)  \equiv
            \int_{\RR^d} \int_{\RR^d} |x-y|^{-s} \, d\mu(y) \, d\mu(x)
            = \int_{\RR^d} k_s * \mu(x) \, d\mu(x),
        \]
        where $k_s$ is the Riesz kernel: 
        $k_s(x)=|x|^{-s}$, $x\in \mathbb{R}^d \setminus \{0\}$.
    \end{definition}
    The study of the Frostman measure is often associated with the Riesz energy via the following proposition.
    \begin{proposition}
        Let $\mu \in\mathcal{M}(\RR^d)$ be an s-Frostman. Then $\mathcal{I}_t(\mu)\lesssim_{s,t,d} \mu(\RR^d), \forall t\in[0,s)$.
    \end{proposition}
Finally, we relate the Riesz energy and the negative Sobolev norm as follows
        \begin{proposition}
        Let $\mu \in \mathcal{M}(\mathbb{R}^d)$ and $s\in (0,d)$. Then we have
        \[
            \mathcal{I}_s(\mu)=\gamma(d,s)
            \int_{\RR^d} |\widehat{\mu}(\xi)|^2 |\xi|^{s-d} \, d\xi
            \sim \gamma(d,s) \|\mu\|^2_{\cH^{s-d}(\RR^d)},
        \]
        where $\|f\|^2_{\cH^{\sigma}(\RR^d)}:=\int_{\RR^d} |\widehat{f}(\xi)|^2 (1+|\xi|^2)^{\sigma/2} \, d\xi$.
    \end{proposition}
So far, we have only discussed the Frostman measures $\mu$ in $\RR^d$. However, if $\text{spt}(\mu) \subset [0,1) \cong \TT$, similar results in this subsection also hold if we consider the Fourier coefficients of $\mu$ instead of the Fourier transform of $\mu$. We will frequently use the following Proposition \ref{Norm Equivalence} as the conclusion of this subsection.
\begin{proposition} \label{Norm Equivalence}
    Let $s \in(0,1]$, $t \in (0,s)$, and $\mu$ be an s-Frostman measure with $\text{spt}(\mu) \subset \TT$. Then we have
    \[
     \|\mu\|^2_{\cH^{t-1}(\TT)} \lesssim_{s,t}1.
    \]
\end{proposition}

\section{Sobolev smoothing inequality}
The purpose of this section is to prove the Sobolev smoothing inequality, Theorem \ref{Thm_Sobolev}.
Relevant Sobolev smoothing inequalities have appeared in \cite{KMPW24} and \cite{DR24}, and are motivated by \cite{P19}. 
The current context, Theorem \ref{Thm_Sobolev}, is more similar to \cite[Theorem 2.1]{P19}, and can be viewed as a generalization to the continuous setting. However, several differences and difficulties are associated with such generalization, which we will explain in more detail in the following subsections. The overall idea still follows from \cite{P19}, but overcoming the technicalities for such generalization is the main contribution of this section.

\subsection{PET induction}
The goal of this subsection is to prove the following version of PET induction, which is suitable for further analysis later.
\begin{theorem} \label{Thm_PET_Gowers}
    Let $\delta \in (0,1]$, $ k \in \NN$, and $N \geq1$. Let $\mathcal{P}= \{P_1, \cdots, P_{k}\in \RR[y]\}$ be a collection of polynomials such that $1 \leq \text{deg}(P_1) < \cdots < \text{deg}(P_k)$. Then there exists $s=s(\mathcal{P})\in \NN$ such that
    \[
    \left|\Lambda_{\mathcal{P};N}(f_0,...,f_k)  \right| ^{O_{\mathcal{P}}(1)}\lesssim_{\mathcal{P}}
        \|f_k\|_{{U^s}(\TT)} ,
    \]
     for any 1-bounded functions $f_0, \cdots, f_k$.
\end{theorem}
\begin{remark}
    There are two distinctions between the finite field setting and the current continuous setting. The first difference is that in the current continuous setting, the $y$-variable lies in an interval, which is only an approximate group. This causes additional technicalities when applying the Cauchy-Schwarz inequality. The second difference is that in the conclusion of Theorem \ref{Thm_PET_Gowers}, there is no error term in the upper bound, which is essential for studying the Sobolev smoothing inequality of the unbounded "functions" in Section \ref{section 4}.
\end{remark}
 To deal with the approximate group structure of an interval, the needed machinery had been developed in \cite{KMPW24}, and we can use their results as a black box. We first introduce a variant of the Gowers norm.

\begin{definition}[Gowers $U^s$-box norm]
    Given any interval $I \subset \mathbb{R}$, define the Fejér kernel
    \[
    \kappa_I(x)=\frac{1}{|I|^2} 1_I \ast 1_{-I}(x),
    \] 
    and set
    \[
    d\nu_I(x):=\kappa_I(x) \, dx.
    \]
    Then for any integers $s\geq 1$, and intervals $I_1,...I_s$, the Gowers $U^s$-box norm of $f$ is defined by
    \[
    \|f\|_{\Box^s_{I_1,...I_s}(\TT)} ^{2^s} :=  \int_{\mathbb{R}^{s}} \int_{\TT}
    \triangle_{h_1,...,h_s}f(x)\, dx \,d\nu_{I_1}(h_1), ...,  d\nu_{I_s}(h_s).
    \]
\end{definition}

Here we record a variant PET induction result proved in \cite[Theorem 6.10]{KMPW24}

\begin{theorem} \label{Thm_PET_Box}
    Let $\delta \in (0,1]$, $ k \in \NN$, and $N \geq1$. Let $\mathcal{P}= \{P_1, \cdots, P_{k}\in \RR[y]\}$ be a collection of polynomials such that $1 \leq \text{deg}(P_1) < \cdots < \text{deg}(P_k)$, and let $f_0, \cdots, f_k$ be 1-bounded. Suppose that
    \[
    \delta \leq
    \left|\Lambda_{\mathcal{P};N}(f_0,...,f_k)  \right| .
    \]
    Then there exists\footnote{In practice, we may assume $s\geq 2$ by the monotonicity of the Gowers $U^s$-box norm.} $s=s(\mathcal{P})\in \NN$ such that
    \[
    \delta^{O_{\mathcal{P}}(1)}
    \lesssim_{\mathcal{P}}
    \|f_k\|_{\Box^s_{[H_1], \cdots,[H_s]}(\TT)},
    \]
    where $H_i \simeq \delta^{O_{\mathcal{P}}(1)} N^{\text{deg}(P_k)}$ for $i= 1, \cdots, s$.
\end{theorem}

\begin{remark}
    We explain here the differences between Theorem \ref{Thm_PET_Box} and \cite[Theorem 6.10]{KMPW24}. Firstly, our definition of the Gowers $U^s$-box norm differs from the one defined in \cite{KMPW24} by the ambient space. In our setting, the ambient space is $\TT$, while in \cite{KMPW24}, the ambient space is, for example, $\RR$. However, we can use an essentially identical argument as in \cite[Theorem 6.10]{KMPW24} to prove Theorem \ref{Thm_PET_Box}. In fact, the argument in our setting can be further simplified since $\TT$ is an abelian compact group. Secondly, the original statement in \cite[Theorem 6.10]{KMPW24} holds for a broader class of collections of polynomials, and the authors pointed out in \cite[Remark 6.9]{KMPW24} that collections of polynomials with distinct degrees satisfy the condition in \cite[Theorem 6.10]{KMPW24}. Finally, the assumption of distinct degrees is essential for the current formulation, which is the primary reason we focus on a collection of polynomials with distinct degrees. One can likely prove a different version of the PET induction for a collection of linearly independent polynomials, but it requires additional ideas than those presented in \cite{KMPW24}. We hope to address this issue in the future.
\end{remark}

Finally, we need to pass from the Gowers $U^s$-box norm to the Gowers $U^s$-norm, and this is achieved in the following lemma. A similar idea has been presented in \cite[Lemma 5.1]{KMPW24}, albeit in a different context.

\begin{lemma} \label{Lma_Box_to_Gowers}
    Let $M \in \NN$, and $0 < H_i < M$, for $i= 1, \cdots, s$. Let $f $ be 1-bounded. Then
    \[
    \|f\|^{2^s}_{\Box^s_{[H_1], \cdots,[H_s]}(\TT)}
    \lesssim \left(\frac{M^s}{\prod_{i=1}^s H_i} \right) \|f\|^{2^s}_{U^s(\TT)}
    \]
    for any $s \geq 2$.
\end{lemma}
\begin{proof}[Proof of Lemma \ref{Lma_Box_to_Gowers}]
Set
\[
F(h_1,...,h_s):=1_{\{h_i \in [-M,M]: \, i=1,\cdots,s\}}(h_1,...,h_s)\cdot
\int_{\TT} \triangle_{h_1, \cdots, h_s}f(x) \,dx .
\]
Then we have
\[
\|f\|^{2^s}_{\Box^s_{[H_1], \cdots,[H_s]}(\TT)}=
\int_{\RR^{s}} F(h_1,...,h_s) \,d\nu_{[H_1]}(h_1), ...,  d\nu_{[H_s]}(h_s)
\]
\[
=\int_{\RR^{s}} \widehat{F}(\xi_1,...,\xi_s) \overline{\widehat{\nu_{[H_1]}}(\xi_1)} \cdots \overline{\widehat{\nu_{[H_s]}}(\xi_s)} \, d\xi_1, ...,\,d\xi_s,
\]
where the last line follows from Parseval's identity. Therefore,
\[
\|f\|^{2^s}_{\Box^s_{[H_1], \cdots,[H_s]}(\TT)} \leq 
\prod_{i=1}^s \|\widehat{\nu_{[H_i]}}\|_{\cL^1(\RR)} \cdot \sup_{\xi_1,\cdots,\xi_s \in \RR} \left| \widehat{F}(\xi_1,...,\xi_s) \right|.
\]
First, note that 
\[
\|\widehat{\nu_{[H_i]}}\|_{\cL^1(\RR)}=\frac{1}{H_i^2}\|\widehat{1_{[H_i]}}\|^2_{\cL^2(\RR)}=\frac{1}{H_i^2}\|1_{[H_i]}\|^2_{\cL^2(\RR)}=\frac{1}{H_i}.
\]
Next, if we set
\[
f_{\underline{w}}(x)=
\left\{
\begin{matrix}
    e\left((\sum_{i=1}^s \xi_i)x\right)f(x), &\text{if}\, w_j=0, \forall j\in \{1,\cdots,s\}. \\
    e(\xi_ix)f(x), &\text{if}\, w_i=1,\, \text{and}\, w_j=0, \forall j\in \{1,\cdots,s\}-\{i\}.\\
    f(x), &\text{otherwise.}\\
\end{matrix}
\right.
\]
Then, we have
\[
\left|
\widehat{F}(\xi_1,...,\xi_s) \right|
= \left|
\int_{[-M,M]^s} \int_{\TT} e(-\sum_{i=1}^s \xi_i h_i) \cdot\prod_{\underline{w}\in\{0,1\}^s} \cC^{|\underline{w}|}f(x+\underline{w}\cdot \underline{h})  \,dx \, dh_1,...,\,dh_s \right|
\]
\[
=\left| \int_{[-M,M]^s} \int_{\TT}  \prod_{\underline{w}\in\{0,1\}^s} \cC^{|\underline{w}|}f_{\underline{w}}(x+\underline{w}\cdot \underline{h})  \,dx \, dh_1,...,\,dh_s\right|
\]
\[
\leq \sum_{i_1=0}^{2M-1} \cdots \sum_{i_s=0}^{2M-1}
\left| \int_{-M+i_1}^{-M+i_1+1}\cdots \int_{-M+i_s}^{-M+i_s+1}
\int_{\TT}  \prod_{\underline{w}\in\{0,1\}^s} \cC^{|\underline{w}|}f_{\underline{w}}(x+\underline{w}\cdot \underline{h})  \,dx \, dh_1,...,\,dh_s\right|
\]
\[
=(2M)^s \cdot \left| 
\int_{\TT^{s+1}}  \prod_{\underline{w}\in\{0,1\}^s} \cC^{|\underline{w}|}f_{\underline{w}}(x+\underline{w}\cdot \underline{h})  \,dx \, dh_1,...,\,dh_s\right|
\]
\[
\leq (2M)^s \cdot \prod_{\underline{w}\in\{0,1\}^s} \|f_{\underline{w}}\|_{U^s(\TT)},
\]
where we use the Gowers-Cauchy-Schwarz inequality \eqref{Gowers-Cauchy-Schwarz inequality} in the last step.\\
Finally, note that when $s\geq 2$, we have
\[
\|f_{\underline{w}}\|_{U^s(\TT)}= \|f\|_{U^s(\TT)},
\]
for every $\underline{w}\in\{0,1\}^s$, since $\|e(\xi\cdot)\|_{U^s(\TT)}=1$ for any $\xi \in \RR$. Combining everything so far, we have shown that
\[
    \|f\|^{2^s}_{\Box^s_{[H_1], \cdots,[H_s]}(\TT)}
    \leq \left(\frac{(2M)^s}{\prod_{i=1}^s H_i} \right) \|f\|^{2^s}_{U^s(\TT)},
\]
and hence, the proof is complete.
\end{proof}

Now, we are ready to prove Theorem \ref{Thm_PET_Gowers}.
\begin{proof}[Proof of Theorem \ref{Thm_PET_Gowers}]
    We set
    \[
    \delta =
    \left|\Lambda_{\mathcal{P};N}(f_0,...,f_k)  \right|,
    \]
    and assume that $ \delta \neq 0$; otherwise, the conclusion is obvious. By Theorem \ref{Thm_PET_Box}, there exists $s\geq 2$ such that
    \[
    \delta^{O_{\mathcal{P}}(1)}
    \lesssim_{\mathcal{P}}
    \|f_k\|_{\Box^s_{[H_1], \cdots,[H_s]}(\TT)},
    \]
    where $H_i \simeq \delta^{O_{\mathcal{P}}(1)} N^{\text{deg}(P_k)}$ for $i= 1, \cdots, s$. Next, by Lemma \ref{Lma_Box_to_Gowers},
     \[
    \|f_k\|^{2^s}_{\Box^s_{[H_1], \cdots,[H_s]}(\TT)}
    \lesssim \left(\frac{N^{s\cdot\text{deg}(P_k)}}{\prod_{i=1}^s H_i} \right) \|f_k\|^{2^s}_{U^s(\TT)},
    \]
    where we choose $M \sim N^{\text{deg}(P_k)}$, since $H_i \lesssim N^{\text{deg}(P_k)}$. Combining the two inequalities above, we have
    \[
    \delta^{O_{\mathcal{P}}(1)} \lesssim_{\mathcal{P}}  \left(\frac{N^{s\cdot\text{deg}(P_k)}}{\prod_{i=1}^s H_i} \right) \|f_k\|^{2^s}_{U^s(\TT)}.
    \]
    Therefore,
    \[
    \delta^{O_{\mathcal{P}}(1)} \sim \delta^{O_{\mathcal{P}}(1)} \cdot \left(\frac{N^{s\cdot\text{deg}(P_k)}}{\prod_{i=1}^s H_i} \right)^{-1} \lesssim_{\mathcal{P}} \|f_k\|^{2^s}_{U^s(\TT)},
    \]
    and this completes the proof.
\end{proof}

\subsection{Preparatory lemmas}
In this subsection, we introduce the dual function, which plays a crucial role in the study of the polynomial Szemerédi theorem, and present several useful lemmas that will be utilized throughout this section.

Define the i-th dual function as follows:
\[
    F^i(x)=F^i_{\mathcal{P};N}(x):= \frac{1}{N} \int_0^N \prod_{\substack{j=0,\cdots, k \\ j \neq i}} f_j(x+P_j(y)-P_i(y)) \,dy,
\]
and set
    \[
    F^i_y(x):=\prod_{\substack{j=0,\cdots, k \\ j \neq i}} f_j(x+P_j(y)-P_i(y)),
    \]
    where we denote $P_0(y)=0$ for the sake of simplicity.
From the definition of the dual function, we have 
\begin{equation} \label{identity for the pairing}
    \Lambda_{\mathcal{P};N}(f_0,...,f_k)= \int_{\TT} f_i(x)F^i(x) \, dx.
\end{equation}
The following lemma is known as the dualization step, which replaces the use of the Hahn-Banach decomposition in \cite{P19}. To the author's knowledge, this idea first appeared in \cite{PP24}. The proof is an easy consequence of the Cauchy-Schwarz inequality when applied to \eqref{identity for the pairing}.
\begin{lemma} \label{Lma_Dual control}
    Let $k \in \NN$. Then for any $i \in \{0,\cdots,k\}$,
    \[
    |\Lambda_{\mathcal{P};N}(f_0,...,f_k)|^2 \leq 
    \Lambda_{\mathcal{P};N}(f_0,...,\overline{F^i},...,f_k),
    \]
    for any 1-bounded functions $f_0, \cdots, f_k$.
\end{lemma}

The following lemma, the dual difference interchange step, is one of the essential discoveries in \cite{P19}. We will use the following continuous version, which is almost identical to the original proof in the finite field setting, as in \cite[Lemma 5.1]{P19}, with repeated applications of the Cauchy-Schwarz inequality. 

\begin{lemma} \label{Lma_Dual Difference Interchange}
    Let $F_y(x)=F(x,y)\in \cL(\TT \times \RR)$ be 1-bounded.
    Set
    \[
    F(x):= \frac{1}{N} \int_0^N F_y(x) \, dy.
    \]
    Then for all $s \geq 2$,
    \[
    \|F\|^{2^{2s-2}}_{U^s(\TT)} \leq 
    \int_{\TT^{s-2}} \| \frac{1}{N} \int_0^N \triangle_{h_1, \cdots, h_{s-2}}F_y(\cdot) \, dy \|^4_{U^2(\TT)} \, dh_1,..., dh_{s-2}.
    \]
\end{lemma}

The final lemma below distinguishes between the finite field setting and the continuous setting, playing another vital role in proving the Sobolev inequality of the unbounded "functions" in Section \ref{section 4}.
The current formulation is an adaptation of \cite[Lemma 4]{DR24}.

\begin{lemma} \label{Lma_PJ}
    Let $s \in \NN$. Then for any $\sigma>0$, there exists $c=c(s, \sigma)>0$ such that
    \[
    \int_{\TT^{s}} \|\triangle_{h_1, \cdots, h_{s}}f\|^{2}_{\cH^{-\sigma}(\TT)} \, dh_1,...,dh_{s}
    \lesssim \|f\|^c_{U^{s+1}(\TT)},
    \]
    for any 1-bounded function $f$.
\end{lemma}

\subsection{Intermediate induction step} \label{Intermediate induction step}
This subsection presents an intermediate induction step that appears in the degree-lowering method we will prove in the later subsection. The starting point of the degree-lowering method, and hence, the intermediate induction step here, comes from the Gowers norm control of the counting operators. However, the only available Gowers norm control is from Theorem \ref{Thm_PET_Gowers}, which only gives control for a collection of polynomials with distinct degrees, and this causes extra difficulties when one wants to run the induction. Therefore, we present a different approach to obtaining Gowers norm control for the intermediate induction step, bypassing the use of PET induction. The following proposition presents the intermediate induction step, and we can view it as a special case of Theorem \ref{Thm_Sobolev}.

\begin{proposition} \label{Prop_Intermediate induction step}
    Let $k,l \in \NN$ with $l \leq k$, $N\geq1$, and let $f_0, \cdots, f_{k}$ be 1-bounded.  Let $\mathcal{P}= \{P_1, \cdots, P_{k}\in \RR[y]\}$ be a collection of polynomials with zero constant terms such that $1 \leq \text{deg}(P_1) < \cdots < \text{deg}(P_{k})$, and let $\mathcal{P}_{k-1} \subset \mathcal{P}$ be any subset with $|\mathcal{P}_{k-1}|=k-1$.
    Let $I_l \subset \{1, \cdots, k\}$ be any subset with $|I_l|=l\geq 1$. For $i \in I_l$, set $f_i(x)=e(\xi_i x)$ for some $\xi \in \ZZ-\{0\}$.\\ Suppose that Theorem \ref{Thm_Sobolev} holds for all $\mathcal{P}_{k-1}\subset \mathcal{P}$. Then there exist $\sigma=\sigma(\mathcal{P})>0$ and $C=C(\mathcal{P}), c=c(\mathcal{P})>0$ such that
        \[
        \left|\Lambda_{\mathcal{P};N}(f_0,...,f_{k}) \right| \lesssim_{\mathcal{P}}
        N^{-C} \cdot \min_{\substack{j=0,\cdots, k \\ j \notin I_l} }
        \|f_j\|^{c}_{\cH^{-\sigma}(\TT)}.
        \]
\end{proposition}

We need another induction to prove Proposition \ref{Prop_Intermediate induction step}. For the sake of clarity, we extract the base case of this induction as the following lemma.

\begin{lemma} \label{Lma_Intermediate induction step}
    Let $k \in \NN$, $N\geq1$, and let $f_0$ be 1-bounded. Let $\mathcal{P}= \{P_1, \cdots, P_{k}\in \RR[y]\}$ be a collection of linearly independent polynomials with zero constant terms. For $i=1 ,\cdots, k$, set $f_i(x)=e(\xi_i x)$ for some $\xi_i \in \ZZ-\{0\}$. Then there exist $\sigma=\sigma(\mathcal{P})>0$ and $C=C(\mathcal{P}), c=c(\mathcal{P})>0$ such that
        \[
        \left|\Lambda_{\mathcal{P};N}(f_0,...,f_{k}) \right| \lesssim_{\mathcal{P}}
        N^{-C} \cdot
        \|f_0\|^{c}_{\cH^{-\sigma}(\TT)}.
        \]
\end{lemma}

We postpone the proof of Lemma \ref{Lma_Intermediate induction step} until the end of this subsection, and we first start with the proof of Proposition \ref{Prop_Intermediate induction step}.

\begin{proof}[Proof of Proposition \ref{Prop_Intermediate induction step} assuming Lemma \ref{Lma_Intermediate induction step}]
We will proceed by downward induction on the parameter $l$, while $k$ is a fixed integer. The base case is when $l=k$, which is the conclusion of Lemma \ref{Lma_Intermediate induction step}, since a collection of polynomials with distinct degrees is also a collection of linearly independent polynomials.  From now on, we will assume that $2\leq k$, $1\leq l< k$, and that the conclusion of Proposition \ref{Prop_Intermediate induction step} holds for $l+1$.\\
We begin with a simple observation, providing a Gowers norm control without relying on the PET induction in this specific situation. We claim that 
\begin{equation} \label{Special_U^2 control}
    \left|\Lambda_{\mathcal{P};N}(f_0,...,f_{k}) \right|^{O_{\mathcal{P}}(1)} \lesssim_{\mathcal{P}}
    \min_{\substack{j=0,\cdots, k \\ j \notin I_l} }\|f_j\|_{U^2(\TT)}.
\end{equation}
To see this, we apply the Cauchy-Schwarz inequality in the $y$ variable to $\Lambda_{\mathcal{P};N}(f_0,...,f_{k})$:
\[
\left|\Lambda_{\mathcal{P};N}(f_0,...,f_{k}) \right|^2=
\left|\frac{1}{N} \int_0^N \int_{\TT} \prod_{\substack{j=0,\cdots, k }} f_j(x+P_j(y)) \,dx \, dy \right|^2 
\]
\[
\leq 
\left|\frac{1}{N} \int_0^N \int_{\TT} \int_{\TT} \prod_{\substack{j=0,\cdots, k }} f_j(x+P_j(y)) \overline{f_j(x'+P_j(y))} \,dx \, dx' \, dy \right|,
\]
and then the change of variable from $x'$ to $x+h$
\[
=\left|\frac{1}{N} \int_0^N \int_{\TT} \int_{\TT} \prod_{\substack{j=0,\cdots, k }} \triangle_hf_j(x+P_j(y))  \,dx \, dh \, dy \right|,
\]
and finally, by Fubini and the triangle inequality
\[
\leq \int_{\TT}\left|\frac{1}{N} \int_0^N \int_{\TT}  \prod_{\substack{j=0,\cdots, k }} \triangle_hf_j(x+P_j(y))  \,dx \, dy \right| \, dh
\]
\[
=\int_{\TT} \left|\Lambda_{\mathcal{P};N}(\triangle_hf_0,...,\triangle_hf_{k}) \right| \, dh
\]
\[
\leq 
\left(  \int_{\TT} \left|\Lambda_{\mathcal{P};N}(\triangle_hf_0,...,\triangle_hf_{k}) \right|^2 \, dh \right)^{1/2},
\]
where the last inequality follows from the Cauchy-Schwarz inequality.
By our assumption, for $i \in I_l$, $f_i(x)=e(\xi_i x)$ for some $\xi \in \ZZ-\{0\}$; and hence, $\triangle_hf_i\equiv1$. Next, by assuming that Theorem \ref{Thm_Sobolev} holds for all $\mathcal{P}_{k-1}\subset \mathcal{P}$, the conclusion of Theorem \ref{Thm_Sobolev} holds for $\{P_j\}_{ j \in I^c_l}\subset \mathcal{P}$, where $I^c_l := \{1,\cdots,k\}-I_l$ with $|I^c_l|\leq k-1$. Therefore, there exist $\sigma'=\sigma'(\mathcal{P})>0$ and $C'=C'(\mathcal{P}), c'=c'(\mathcal{P})>0$ such that
    \[
     \Lambda_{\mathcal{P};N}(\triangle_hf_0,...,\triangle_hf_{k})
    =
        \prod_{\substack{j=0,\cdots, k \\ j \notin I_l}} \int_{\TT} \triangle_{h} f_j + O_{\mathcal{P}} \left(
        N^{-C'} \cdot \min_{\substack{j=0,\cdots, k \\ j \notin I_l}}  \|\triangle_{h}f_j\|^{c'}_{\cH^{-\sigma'}(\TT)}
        \right).
    \]
Then, we have
    \begin{equation} \label{Intermediate_U^2_Degree lowering_0}
        \left|\Lambda_{\mathcal{P};N}(f_0,...,f_{k})  \right| ^{O(1)}\lesssim_{\mathcal{P}} (I)+(II),
    \end{equation}
    where
    \begin{equation} \label{Intermediate_U^2_Degree lowering_I}
        (I):=  \int_{\TT}\left| \prod_{\substack{j=0,\cdots, k \\ j \notin I_l}} \int_{\TT} \triangle_{h} f_j \right|^2 \, dh,
    \end{equation}
    \begin{equation} \label{Intermediate_U^2_Degree lowering_II}
        (II):=N^{-C'} \cdot \int_{\TT} \min_{\substack{j=0,\cdots, k \\ j \notin I_l}}  \|\triangle_{h}f_j\|^{2c'}_{\cH^{-\sigma'}(\TT)} \,dh.
    \end{equation}
    For the first term \eqref{Intermediate_U^2_Degree lowering_I}, by 1-boundedness of $f_j$ and the triangle inequality, we have
    \begin{equation} \label{Intermediate_U^2_Degree lowering_I(conclusion)}
        (I) \leq \min_{\substack{j=0,\cdots, k \\ j \notin I_l} }\int_{\TT}\left|  \int_{\TT} \triangle_{h} f_j \right|^2 \, dh
    = \min_{\substack{j=0,\cdots, k \\ j \notin I_l} } \|f_j\|^4_{U^2(\TT)}.
    \end{equation}
    For the second term \eqref{Intermediate_U^2_Degree lowering_II}, if $2c\geq 2$, we use 1-boundedness of $f_i$ to bound
    \[
    \int_{\TT^{s}} \|\triangle_{h}f_j\|^{2c}_{\cH^{-\sigma}(\TT)}  \, dh \leq
    \int_{\TT^{s}} \|\triangle_{h}f_j\|^{2}_{\cH^{-\sigma}(\TT)}  \, dh.
    \]
    If $2c<2$, we apply Hölder's inequality to bound
    \[
    \int_{\TT^{1}} \|\triangle_{h}f_j\|^{2c}_{\cH^{-\sigma}(\TT)}  \, dh \leq
    \left(
    \int_{\TT^{s}} \|\triangle_{h}f_j\|^{2}_{\cH^{-\sigma}(\TT)}  \, dh
    \right)^c.
    \]
    In either case, we have
    \[
    (II)^{O_{\mathcal{P}}(1)}\lesssim_{\mathcal{P}}
    \min_{\substack{j=0,\cdots, k \\ j \neq i} }
    \int_{\TT^{s-1}} \|\triangle_{h_1, \cdots, h_{s-1}}f_j\|^{2}_{\cH^{-\sigma}(\TT)}  \, dh_1,..., dh_{s-1}.
    \]
And hence, by Lemma \ref{Lma_PJ}, we conclude that
\begin{equation} \label{Intermediate_U^2_Degree lowering_II(conclusion)}
    (II)^{O_{\mathcal{P}}(1)}\lesssim_{\mathcal{P}} N^{-C'} \cdot
    \min_{\substack{j=0,\cdots, k \\ j \notin I_l} }
        \|f_j\|_{{U^{2}}(\TT)}
        \lesssim
    \min_{\substack{j=0,\cdots, k \\ j \notin I_l} }
        \|f_j\|_{{U^{2}}(\TT)}.
\end{equation}
Combining \eqref{Intermediate_U^2_Degree lowering_0}, \eqref{Intermediate_U^2_Degree lowering_I(conclusion)}, and \eqref{Intermediate_U^2_Degree lowering_II(conclusion)} gives the proof of the claim \eqref{Special_U^2 control}.\\
Next, let's move back to finish the induction. Choose any $i \in \{1,\cdots,k\}-I_l$ and recall the dual function
    \[
    F^i(x)= \frac{1}{N} \int_0^N \prod_{\substack{j=0,\cdots, k \\ j \neq i}} f_j(x+P_j(y)-P_i(y)) \,dy.
    \]
    By Lemma \ref{Lma_Dual control}, \eqref{Special_U^2 control}, and \eqref{U^2-inverse theorem} we have
 \begin{equation} \label{Induction_U^2_Degree lowering_0}
        \left|\Lambda_{\mathcal{P};N}(f_0,...,f_{k})  \right| ^{O_{\mathcal{P}}(1)}\lesssim_{\mathcal{P}} (I')+(II'),
    \end{equation}
    where
    \begin{equation} \label{Induction_U^2_Degree lowering_I}
        (I'):= \left| \frac{1}{N} \int_0^N \int_{\TT} F^i_y(x) \,dx \, dy \right| ,
    \end{equation}
    \begin{equation} \label{Induction_U^2_Degree lowering_II}
        (II'):=\sup_{\xi_i \in \ZZ-\{0\}} \left| \frac{1}{N} \int_0^N \int_{\TT}e(\xi_i x) F^i_y(x)  \,dx \, dy \right| .
    \end{equation}

For the first term \eqref{Induction_U^2_Degree lowering_I}, by our induction hypothesis that Theorem \ref{Thm_Sobolev} holds for  $\mathcal{P}^i_{k-1}:= \mathcal{P}-\{i\}$, there exist $\sigma''=\sigma''(\mathcal{P})>0$ and $C''=C''(\mathcal{P}), c''=c''(\mathcal{P})>0$ such that
    \[
     \left|\frac{1}{N} \int_0^N \int_{\TT} F^i_y(x) \,dx \, dy  \right|
    =
     \left|\frac{1}{N} \int_0^N \int_{\TT} \prod_{\substack{j=0,\cdots, k \\ j \neq i}} f_j(x+P_j(y)) \,dx \, dy \right|
    \]
    \[
    \lesssim_{\mathcal{P}} 
        N^{-C''} \cdot \min_{\substack{j=0,\cdots, k \\ j \notin I_l \bigcup \{i\}}}  \|f_j\|^{c''}_{\cH^{-\sigma''}(\TT)}
    \]
    where we use the fact that $ \int_{\TT} f_{i'}=\int_{\TT} e(\xi_{i'} x) \, dx=0$ for any $i' \in I_l$.\\
    For the second term \eqref{Induction_U^2_Degree lowering_II}, by our induction hypothesis that the conclusion of Proposition \ref{Prop_Intermediate induction step} holds for $l+1$, there exist $\sigma'''=\sigma'''(\mathcal{P})>0$ and $C'''=C'''(\mathcal{P}), c'''=c'''(\mathcal{P})>0$ such that
    \[
    \left| \frac{1}{N} \int_0^N \int_{\TT}e(\xi_i x) F^i_y(x)  \,dx \, dy \right|
    \]
    \[
    =\left|
     \frac{1}{N} \int_0^N \int_{\TT}e(\xi_i (x+P_i(y))) \prod_{\substack{j=0,\cdots, k \\ j \neq i}} f_j(x+P_j(y)) \,dx \, dy \right|
    \]
    \[
    \lesssim_{\mathcal{P}}
        N^{-C'''} \cdot \min_{\substack{j=0,\cdots, k \\ j \notin I_l \bigcup \{i\}} }
        \|f_j\|^{c'''}_{\cH^{-\sigma'''}(\TT)}.
    \]
    If $l<k-1$, we can run through the above argument for all $i \in \{1,\cdots,k\}-I_l$ to upgrade the control of $\cH^{-\sigma}$-norm from $j \notin I_l \bigcup \{i\}$ to $j \notin I_l $. If $l=k-1$, we have to deal with the case for $F^0(x)$ in addition. However, we can estimate \eqref{Induction_U^2_Degree lowering_I} and \eqref{Induction_U^2_Degree lowering_II} by Lemma \ref{Lma_Intermediate induction step} directly in this situation. Therefore, we can also upgrade the control of $\cH^{-\sigma}$-norm from $j \notin I_l \bigcup \{i\}$ to $j \notin I_l $ in this case.\\
    To conclude, combining \eqref{Induction_U^2_Degree lowering_0}, and the estimates for \eqref{Induction_U^2_Degree lowering_I} and \eqref{Induction_U^2_Degree lowering_II} closes the induction.
\end{proof}

Finally, it remains to prove Lemma \ref{Lma_Intermediate induction step}.

\begin{proof}[Proof of Lemma \ref{Lma_Intermediate induction step}]
    In this case, we have
    \[
    \Lambda_{\mathcal{P};N}(f_0,...,f_{k})=
    \frac{1}{N} \int_0^N \int_{\TT}f_0(x) e((\sum_{i=1}^k \xi_i)x+\sum_{i=1}^k \xi_i P_i(y)) \, dx \,dy
    \]
    \[
    =\widehat{f_0}(-\sum_{i=1}^k \xi_i) \cdot
    \frac{1}{N} \int_0^N e(\sum_{i=1}^k \xi_i P_i(y))
    \, dy.
    \]
    Denote $P_i(y)=\sum_{j=1}^d a_{ji} y^j$, where $d=\max_{i=1,\cdots,k} \text{deg}(P_i)$. Our goal is to show that
    \begin{equation} \label{control of frequency}
        \left| \sum_{i=1}^k \xi_i \right| \lesssim_{\mathcal{P}}
        \max_{j=1,\cdots,d} \left| \sum_{i=1}^k a_{ji} \xi_i \right|.
    \end{equation}
    and
    \begin{equation} \label{control of frequency (II)}
        1 \leq\sum_{i=1}^k |\xi_i|^2 \lesssim_{\mathcal{P}} \max_{j=1,\cdots,d} \left| \sum_{i=1}^k a_{ji} \xi_i \right|^2.
    \end{equation}
    Note that by the linear independence assumption of $\mathcal{P}$, we have $d \geq k$, and there exists a subindex $J=\{j_1,\cdots,j_k\}\subset \{1,\cdots,d\}$ with $|J|=k$ such that the matrix
    \[
    A:= \begin{pmatrix}
        a_{ji}
    \end{pmatrix}_{\substack{j\in J \\ i=1,\cdots,k}}
    \]
    is invertible. Then we denote
    \begin{equation} \label{Matrix}
        A \begin{pmatrix} \xi_1 \\ \vdots \\ \xi_k \end{pmatrix}=
    \begin{pmatrix} \eta_{j_1} \\ \vdots \\ \eta_{j_k} \end{pmatrix},
    \end{equation}
    with
    \[
    \max_{i=1,\cdots,k}|\eta_{j_k}| \leq \max_{j=1,\cdots,d} \left| \sum_{i=1}^k a_{ji} \xi_i \right|.
    \]
    By applying $A^{-1}$ to \eqref{Matrix}, we have
    \[
    \begin{pmatrix} \xi_1 \\ \vdots \\ \xi_k \end{pmatrix}=A^{-1}\begin{pmatrix} \eta_{j_1} \\ \vdots \\ \eta_{j_k} \end{pmatrix},
    \]
    so that
    \[
    \max_{i=1,\cdots,k}|\xi_i| \leq \|A^{-1}\|_{\infty} \sum_{i=1}^k|\eta_i|,
    \]
    where $\|B\|_{\infty}:= \max_{i,j} |b_{i,j}|$ with $B=\begin{pmatrix}
        b_{ij}
    \end{pmatrix}_{i,j}$. Therefore, we have
    \[
    \left| \sum_{i=1}^k \xi_i \right| 
    \leq \sum_{i=1}^k  \left| \xi_i \right|
    \leq k^2 \|A^{-1}\|_{\infty} \cdot \max_{j=1,\cdots,d} \left| \sum_{i=1}^k a_{ji} \xi_i \right|
    \lesssim_{\mathcal{P}}
        \max_{j=1,\cdots,d} \left| \sum_{i=1}^k a_{ji} \xi_i \right|,
    \]
    which proves \eqref{control of frequency}. Next, to prove \eqref{control of frequency (II)}, by Cauchy-Schwarz inequality, we have
    \[
    \left( \sum_{j \in J} c_j \sum_{i=1}^k a_{ji} \xi_i  \right)^2
    \leq \left( \sum_{j \in J} c_j^2  \right) \cdot
    \left( \sum_{j \in J}  \left|\sum_{i=1}^k a_{ji} \xi_i \right|^2 \right).
    \]
    Let $\{e_l\}_{l=1}^k$ be the standard basis in $\RR^k$, that is, the $l$-th coordinate is the only non-zero coordinate. Then, if we choose 
    \[
    \begin{pmatrix}
        c_j
    \end{pmatrix}^{T}= \begin{pmatrix}
        e_l
    \end{pmatrix}^{T}A^{-1},
    \]
    then we get
    \[
    \sum_{j \in J} c_j \sum_{i=1}^k a_{ji} \xi_i =\xi_l.
    \]
    Therefore, this shows that
    \[
    |\xi_l|^2 \lesssim_{\mathcal{P}} \max_{j=1,\cdots,d} \left| \sum_{i=1}^k a_{ji} \xi_i \right|^2
    \]
    holds for all $l=1,\cdots,k$. By summing all $l$ proves \eqref{control of frequency (II)}.\\
    Finally, if $\sum_{i=1}^k \xi_i \neq 0$, by Lemma \ref{Lma_Oscillatory integral} and \eqref{control of frequency}, we have
    \[
    \left|\Lambda_{\mathcal{P};N}(f_0,...,f_{k})\right|
    \]
    \[
    \leq \left|
    \widehat{f_0}(-\sum_{i=1}\xi_i)\right| \cdot \left|
    \frac{1}{N} \int_0^N e(\sum_{i=1}^k \xi_i P_i(y))
    \, dy  \right|
    \]
    \[
    \lesssim_{\mathcal{P}}
    \left|
    \widehat{f_0}(-\sum_{i=1}\xi_i)\right| \cdot 
        \left(  N\max_{j=1,\cdots,d} \left| \sum_{i=1}^k a_{ji} \xi_i \right| \right)^{-1/d}
    \]
    \[
    \lesssim_{\mathcal{P}}
    \left|
    \widehat{f_0}(-\sum_{i=1}\xi_i)\right| \cdot 
        \left(  N|\sum_{i=1}\xi_i| \right)^{-1/d}
    \]
    \[
    \lesssim
        N^{-1/d}
        \cdot
        \|f_0\|_{\cH^{-2/d}(\TT)}.
    \]
    If $\sum_{i=1}^k \xi_i = 0$, by Lemma \ref{Lma_Oscillatory integral} and \eqref{control of frequency (II)}, we have
    \[
    \left|\Lambda_{\mathcal{P};N}(f_0,...,f_{k})\right|
    \]
    \[
    \leq \left|
    \widehat{f_0}(0)\right| \cdot \left|
    \frac{1}{N} \int_0^N e(\sum_{i=1}^k \xi_i P_i(y))
    \, dy  \right|
    \]
    \[
    \lesssim_{\mathcal{P}}
    \left|
    \widehat{f_0}(0)\right| \cdot 
        \left(  N\max_{j=1,\cdots,d} \left| \sum_{i=1}^k a_{ji} \xi_i \right| \right)^{-1/d}
    \]
    \[
    \lesssim_{\mathcal{P}}
    \left|
    \widehat{f_0}(0)\right| \cdot 
      N^{-1/d}
    \]
    \[
    \lesssim
        N^{-1/d}
        \cdot
        \|f_0\|_{\cH^{-\sigma}(\TT)},
    \]
for any $\sigma \geq 0$. Therefore, we may choose $\sigma=2/d$, $C=1/d$, and $c=1$. This completes the proof.
\end{proof}

\subsection{Degree-lowering method}
This subsection aims at generalizing the degree-lowering method proved in \cite{P19} to the current context. Although the proof strategy mostly follows from \cite[Lemma 4.1]{P19}, we in addition implement Lemma \ref{Lma_PJ} to distinguish between the finite field setting and the continuous setting. The following proposition presents the degree-lowering method in the current context.

\begin{proposition} \label{Prop_Degree lowering}
    Let $k \in \NN$, $N\geq1$, and let $f_0, \cdots, f_{k}$ be 1-bounded. Let $\mathcal{P}= \{P_1, \cdots, P_{k}\in \RR[y]\}$ be a collection of polynomials with zero constant terms such that $1 \leq \text{deg}(P_1) < \cdots < \text{deg}(P_{k})$, and let $\mathcal{P}_{k-1} \subset \mathcal{P}$ be any subset with $|\mathcal{P}_{k-1}|=k-1$. Now fix any $i \in \{1,\cdots,k\}$.
    Suppose that the following three statements are true: 
    \begin{enumerate}
        \item Theorem \ref{Thm_Sobolev} holds for all $\mathcal{P}_{k-1}\subset \mathcal{P}$.
        \item There exist $\sigma=\sigma(\mathcal{P})>0$ and $C=C(\mathcal{P}), c=c(\mathcal{P})>0$ such that
        \begin{equation} \label{Degree lowering_intermediate induction}
            \left|\Lambda_{\mathcal{P};N}(f_0,...,f_{k}) \right| \lesssim_{\mathcal{P}}
        N^{-C} \cdot \min_{\substack{j=0,\cdots, k \\ j \neq i} }
        \|f_j\|^{c}_{\cH^{-\sigma}(\TT)},
        \end{equation}
        if $f_i(x)=e(\xi_i x)$ for some $\xi_i \in \ZZ-\{0\}$.
        \item There exists $s=s(\mathcal{P})\geq 2$ such that
        \begin{equation} \label{Degree lowering_Gowers control}
            \left|\Lambda_{\mathcal{P};N}(f_0,...,f_{k})  \right| ^{O_{\mathcal{P}}(1)}\lesssim_{\mathcal{P}}
        \|f_i\|_{{U^{s+1}}(\TT)}.
        \end{equation}
    \end{enumerate}
      Then 
      \[
    \left|\Lambda_{\mathcal{P};N}(f_0,...,f_{k})  \right| ^{O_{\mathcal{P}}(1)}\lesssim_{\mathcal{P}}
    \min_{\substack{j=0,\cdots, k \\ j \neq i} }
        \|f_j\|_{{U^{s}}(\TT)}.
    \]
\end{proposition}

\begin{proof}[Proof of Proposition \ref{Prop_Degree lowering}] 
    Recall the dual function
    \[
    F^i(x)= \frac{1}{N} \int_0^N \prod_{\substack{j=0,\cdots, k \\ j \neq i}} f_j(x+P_j(y)-P_i(y)) \,dy.
    \]
    Note that $F^i_y$ is 1-bounded, since $f_0, \cdots, f_{k}$ are 1-bounded.
    By Lemma \ref{Lma_Dual control} and the assumption \eqref{Degree lowering_Gowers control}, we have the ${U^{s+1}}(\TT)$ control of the dual function
    \begin{equation} \label{Degree lowering_Gowers control for Dual}
        \left|\Lambda_{\mathcal{P};N}(f_0,...,f_{k})  \right| ^{O_{\mathcal{P}}(1)}\lesssim_{\mathcal{P}}
        \|F^i\|_{{U^{s+1}}(\TT)}.
    \end{equation}
    Next, by Lemma \ref{Lma_Dual Difference Interchange} and \eqref{Degree lowering_Gowers control for Dual}, we have
    \[
    \left|\Lambda_{\mathcal{P};N}(f_0,...,f_{k})  \right| ^{O_{\mathcal{P}}(1)}\lesssim_{\mathcal{P}}
     \int_{\TT^{s-1}} \| \frac{1}{N} \int_0^N \triangle_{h_1, \cdots, h_{s-1}}F^i_y(\cdot) \, dy \|^4_{U^2(\TT)} \, dh_1,..., dh_{s-1}.
    \]
     Hence, by \eqref{U^2-inverse theorem}, we have
    \begin{equation} \label{Degree lowering_0}
        \left|\Lambda_{\mathcal{P};N}(f_0,...,f_{k})  \right| ^{O_{\mathcal{P}}(1)}\lesssim_{\mathcal{P}} (I)+(II),
    \end{equation}
    where
    \begin{equation} \label{Degree lowering_I}
        (I):=\int_{\TT^{s-1}} \left| \frac{1}{N} \int_0^N \int_{\TT} \triangle_{h_1, \cdots, h_{s-1}}F^i_y(x) \,dx \, dy \right|^2 \, dh_1,..., dh_{s-1},
    \end{equation}
    \begin{equation} \label{Degree lowering_II}
        (II):=\int_{\TT^{s-1}} \sup_{\xi_i \in \ZZ-\{0\}} \left| \frac{1}{N} \int_0^N \int_{\TT}e(\xi_i x) \triangle_{h_1, \cdots, h_{s-1}}F^i_y(x)  \,dx \, dy \right|^2 \, dh_1,..., dh_{s-1}.
    \end{equation}
    Thus, it remains to estimate the integrals \eqref{Degree lowering_I} and \eqref{Degree lowering_II}.\\
    We first estimate the second integral \eqref{Degree lowering_II}. By our assumption \eqref{Degree lowering_intermediate induction}, there exist $\sigma=\sigma(\mathcal{P})>0$ and $C=C(\mathcal{P}), c=c(\mathcal{P})>0$ such that
    \[
    \left| \frac{1}{N} \int_0^N \int_{\TT}e(\xi_i x) \triangle_{h_1, \cdots, h_{s-1}}F^i_y(x)  \,dx \, dy \right|
    \]
    \[
    =\left|
     \frac{1}{N} \int_0^N \int_{\TT}e(\xi_i x) \prod_{\substack{j=0,\cdots, k \\ j \neq i}} \triangle_{h_1, \cdots, h_{s-1}}f_j(x+P_j(y)-P_i(y)) \,dx \, dy \right|
    \]
    \[
    =\left|
     \frac{1}{N} \int_0^N \int_{\TT}e(\xi_i (x+P_i(y))) \prod_{\substack{j=0,\cdots, k \\ j \neq i}} \triangle_{h_1, \cdots, h_{s-1}}f_j(x+P_j(y)) \,dx \, dy \right|
    \]
    \[
    \lesssim_{\mathcal{P}}
        N^{-C} \cdot \min_{\substack{j=0,\cdots, k \\ j \neq i} }
        \|\triangle_{h_1, \cdots, h_{s-1}}f_j\|^{c}_{\cH^{-\sigma}(\TT)}
    \]
    \[
    \lesssim_{\mathcal{P}}
        \min_{\substack{j=0,\cdots, k \\ j \neq i} }
        \|\triangle_{h_1, \cdots, h_{s-1}}f_j\|^{c}_{\cH^{-\sigma}(\TT)}.
    \]
    Therefore,
    \[
    (II) \lesssim_{\mathcal{P}} \min_{\substack{j=0,\cdots, k \\ j \neq i} }
    \int_{\TT^{s-1}} \|\triangle_{h_1, \cdots, h_{s-1}}f_j\|^{2c}_{\cH^{-\sigma}(\TT)}  \, dh_1,..., dh_{s-1}.
    \]
    If $2c\geq 2$, we use 1-boundedness of $f_i$ to bound
    \[
    \int_{\TT^{s-1}} \|\triangle_{h_1, \cdots, h_{s-1}}f_j\|^{2c}_{\cH^{-\sigma}(\TT)}  \, dh_1,..., dh_{s-1} \leq
    \int_{\TT^{s-1}} \|\triangle_{h_1, \cdots, h_{s-1}}f_j\|^{2}_{\cH^{-\sigma}(\TT)}  \, dh_1,..., dh_{s-1}.
    \]
    If $2c<2$, we apply Hölder's inequality to bound
    \[
    \int_{\TT^{s-1}} \|\triangle_{h_1, \cdots, h_{s-1}}f_j\|^{2c}_{\cH^{-\sigma}(\TT)}  \, dh_1,..., dh_{s-1} \leq
    \left(
    \int_{\TT^{s-1}} \|\triangle_{h_1, \cdots, h_{s-1}}f_j\|^{2}_{\cH^{-\sigma}(\TT)}  \, dh_1,..., dh_{s-1}
    \right)^c.
    \]
    In either case, we have
    \[
    (II)^{O_{\mathcal{P}}(1)}\lesssim_{\mathcal{P}}
    \min_{\substack{j=0,\cdots, k \\ j \neq i} }
    \int_{\TT^{s-1}} \|\triangle_{h_1, \cdots, h_{s-1}}f_j\|^{2}_{\cH^{-\sigma}(\TT)}  \, dh_1,..., dh_{s-1}.
    \]
And hence, by Lemma \ref{Lma_PJ}, we conclude that
\begin{equation} \label{Degree lowering_II(conclusion)}
    (II)^{O_{\mathcal{P}}(1)}\lesssim_{\mathcal{P}}
    \min_{\substack{j=0,\cdots, k \\ j \neq i} }
        \|f_j\|_{{U^{s}}(\TT)}.
\end{equation}
    Next, we estimate the first integral \eqref{Degree lowering_I}. By our assumption that Theorem \ref{Thm_Sobolev} holds for  $\mathcal{P}^i_{k-1}:= \mathcal{P}-\{i\}$, there exist $\sigma'=\sigma'(\mathcal{P})>0$ and $C'=C'(\mathcal{P}), c'=c'(\mathcal{P})>0$ such that
    \[
     \frac{1}{N} \int_0^N \int_{\TT} \triangle_{h_1, \cdots, h_{s-1}}F^i_y(x) \,dx \, dy 
    \]
    \[
    =
     \frac{1}{N} \int_0^N \int_{\TT} \prod_{\substack{j=0,\cdots, k \\ j \neq i}} \triangle_{h_1, \cdots, h_{s-1}}f_j(x+P_j(y)-P_i(y)) \,dx \, dy 
    \]
    \[
    =
     \frac{1}{N} \int_0^N \int_{\TT} \prod_{\substack{j=0,\cdots, k \\ j \neq i}} \triangle_{h_1, \cdots, h_{s-1}}f_j(x+P_j(y)) \,dx \, dy 
    \]
    \[
    =
        \prod_{\substack{j=0,\cdots, k \\ j \neq i}} \int_{\TT} \triangle_{h_1, \cdots, h_{s-1}} f_i + O_{\mathcal{P}} \left(
        N^{-C'} \cdot \min_{\substack{j=0,\cdots, k \\ j \neq i}}  \|\triangle_{h_1, \cdots, h_{s-1}}f_j\|^{c'}_{\cH^{-\sigma'}(\TT)}
        \right).
    \]
We can estimate the second term as above for \eqref{Degree lowering_II}.
Therefore, it suffices to estimate the integral
\[
\int_{\TT^{s-1}} \left| \prod_{\substack{j=0,\cdots, k \\ j \neq i}} \int_{\TT} \triangle_{h_1, \cdots, h_{s-1}} f_j \right|^2 \, dh_1,..., dh_{s-1},
\]
By 1-boundedness of $f_i$, we have
\[
\int_{\TT^{s-1}} \left| \prod_{\substack{j=0,\cdots, k \\ j \neq i}} \int_{\TT} \triangle_{h_1, \cdots, h_{s-1}} f_j \right|^2 \, dh_1,..., dh_{s-1}
\]
\[
\leq \min_{\substack{j=0,\cdots, k \\ j \neq i}} \int_{\TT^{s-1}} \left| \int_{\TT} \triangle_{h_1, \cdots, h_{s-1}} f_j \right|^2 \, dh_1,..., dh_{s-1}
\]
\[
= \min_{\substack{j=0,\cdots, k \\ j \neq i}} \|f_j\|^{2^s}_{U^s(\TT)}.
\]
To conclude, we also have
\begin{equation} \label{Degree lowering_I(conclusion)}
    (I)^{O_{\mathcal{P}}(1)}\lesssim_{\mathcal{P}}
    \min_{\substack{j=0,\cdots, k \\ j \neq i} }
        \|f_j\|_{{U^{s}}(\TT)}.
\end{equation}
Finally, combining inequalities \eqref{Degree lowering_0}, \eqref{Degree lowering_II(conclusion)}, and \eqref{Degree lowering_I(conclusion)} completes the proof.

\end{proof}

\subsection{Proof of the Sobolev smoothing inequality} \label{Proof of Theorem Thm_Sobolev}
Now, we are ready to prove Theorem \ref{Thm_Sobolev}.
\begin{proof}[Proof of Theorem \ref{Thm_Sobolev}]
We will proceed by induction.
We first prove the base case for $k=1$. By applying Parseval's identity to obtain
\[
    \Lambda_{\mathcal{P};N}(f_0,f_1)= \frac{1}{N}
    \int_0^N
    \int_{\TT} f_0(x)f_1(x+P_1(y)) \,dx \,dy
    \]
    \[
    =\int_{\TT}f_0 \cdot \int_{\TT}f_1 + \sum_{\xi \neq 0}\widehat{f_0}(\xi) \widehat{f_1}(-\xi) \left( \frac{1}{N} \int_0^N e(-\xi P_1(y)) \, dy \right).
    \]
    Next, by Lemma \ref{Lma_Oscillatory integral}, we have
    \[
    |\frac{1}{N} \int_0^N e(-\xi P_1(y))\, dy |\lesssim_{\mathcal{P}}(|\xi|N)^{-1/(\text{deg}(P_1))},
    \]
    Finally, we apply the Cauchy-Schwarz inequality to conclude that
    \[
    |\sum_{\xi \neq 0}\widehat{f_0}(\xi) \widehat{f_1}(-\xi) \left( \frac{1}{N} \int_0^N e(-\xi P_1(y)) \, dy \right)| \lesssim_{\mathcal{P}}
    N^{-1/(\text{deg}(P_1))} \cdot \|f_0\|_{\cH^{-\sigma}(\TT)} \cdot \|f_1\|_{\cH^{-\sigma}(\TT)} ,
    \]
    for $\sigma=1/(\text{deg}(P_1))$. Note that by 1-boundedness of $f_i$,
    \[
    \|f_0\|_{\cH^{-\sigma}(\TT)} \cdot \|f_1\|_{\cH^{-\sigma}(\TT)}
    \leq \min_{i=0,1}  \|f_i\|_{\cH^{-\sigma}(\TT)},
    \]
    Therefore, we have shown that
    \begin{equation} \label{two-term case}
        \Lambda_{\mathcal{P};N}(f_0,f_1)= \int_{\TT}f_0 \cdot \int_{\TT}f_1 +O_{\mathcal{P}} \left(
        N^{-1/(\text{deg}(P_1))} \cdot \min_{i=0,1}  \|f_i\|_{\cH^{-\sigma}(\TT)} \right),
    \end{equation}
    and this finishes the proof of Theorem \ref{Thm_Sobolev} for $k=1$.\\
    From now on, we will assume that $k \geq 2$ and that Theorem \ref{Thm_Sobolev} holds for $k-1$. We claim that there exist $\sigma'=\sigma'(\mathcal{P})>0$ and $C'=C'(\mathcal{P}), c'=c'(\mathcal{P})>0$ such that
        \begin{equation} \label{Proof_U^1}
            \left|\Lambda_{\mathcal{P};N}(f_0,...,f_k)  \right| ^{O_{\mathcal{P}}(1)}\lesssim_{\mathcal{P}}
        \min_{i=0,\cdots, k}  \|f_i\|_{U^1(\TT)}+
        N^{-C'} \cdot \min_{i=0,\cdots, k}  \|f_i\|^{c'}_{\cH^{-\sigma'}(\TT)} ,
        \end{equation}
        for any 1-bounded functions $f_0, \cdots, f_k $. We will first use the claim \eqref{Proof_U^1} to finish the proof of Theorem \ref{Thm_Sobolev}, and then prove the claim \eqref{Proof_U^1} later.\\
        For $i \in \{1,\cdots,k \}$, we write
        \[
        \Lambda_{\mathcal{P};N}(f_0,...,f_k)=
        \]
        \begin{equation} \label{Proof_decomposition}
            \Lambda_{\mathcal{P};N}(f_0,...,f_i-\int_{\TT}f_i,...,f_k)+\Lambda_{\mathcal{P};N}(f_0,...,\int_{\TT}f_i,...,f_k).
        \end{equation}
        For the first part of \eqref{Proof_decomposition}, we apply the claim \eqref{Proof_U^1} to conclude that
        \begin{equation} \label{Proof_decomposition_1}
            \left|\Lambda_{\mathcal{P};N}(f_0,...,f_i-\int_{\TT}f_i,...,f_k)  \right| ^{O_{\mathcal{P}}(1)}\lesssim_{\mathcal{P}}
        N^{-C'} \cdot \min_{\substack{j=0,\cdots, k \\ j \neq i} }
        \|f_j\|^{c'}_{\cH^{-\sigma'}(\TT)}.
        \end{equation}
        For the second part of \eqref{Proof_decomposition}, we apply the induction hypothesis for $\mathcal{P}^i_{k-1}:= \mathcal{P}-\{i\}$ to conclude that there exist $\sigma''=\sigma''(\mathcal{P})>0$ and $C''=C''(\mathcal{P}), c''=c''(\mathcal{P})>0$ such that
        \begin{equation} \label{Proof_decomposition_2}
            \Lambda_{\mathcal{P};N}(f_0,...,\int_{\TT}f_i,...,f_k)=
            \prod_{i=0}^k \int_{\TT} f_i  +O_{\mathcal{P}} \left(
        N^{-C''} \cdot \min_{\substack{j=0,\cdots, k \\ j \neq i} }
        \|f_j\|^{c''}_{\cH^{-\sigma''}(\TT)} \right).
        \end{equation}
        Combining \eqref{Proof_decomposition}, \eqref{Proof_decomposition_1}, and \eqref{Proof_decomposition_2} for $i \in \{1,\cdots,k \}$ completes the proof of Theorem \ref{Thm_Sobolev}. Thus, it remains to prove the claim \eqref{Proof_U^1}.\\
        Next, we want to apply Proposition \ref{Prop_Degree lowering} to perform the degree lowering method. The first assumption in Proposition \ref{Prop_Degree lowering} is satisfied by our induction hypothesis. The second assumption is the conclusion of Proposition \ref{Prop_Intermediate induction step} for $(k,l)=(k,1)$, and is true by our induction hypothesis. Lastly, the third assumption initially holds for $f_k$ by Theorem \ref{Thm_PET_Gowers}. By the monotonicity of the Gowers norm \eqref{Monotonicity}, we may assume that 
        \[
    \left|\Lambda_{\mathcal{P};N}(f_0,...,f_k)  \right| ^{O_{\mathcal{P}}(1)}\lesssim_{\mathcal{P}}
        \|f_k\|_{{U^{s+1}}(\TT)} ,
    \]
    for some $s \geq 2$. Therefore, by Proposition \ref{Prop_Degree lowering}, we have
    \[
    \left|\Lambda_{\mathcal{P};N}(f_0,...,f_{k})  \right| ^{O_{\mathcal{P}}(1)}\lesssim_{\mathcal{P}}
    \min_{\substack{j=0,\cdots, k \\ j \neq k} }
        \|f_j\|_{{U^{s}}(\TT)}
        \lesssim_{\mathcal{P}}
    \min_{\substack{j=0,\cdots, k \\ j \neq k} }
        \|f_j\|_{{U^{s+1}}(\TT)},
    \]
    where the last inequality follows again by the monotonicity of the Gowers norm \eqref{Monotonicity}. Since we assume that $k \geq 2$, by choosing any $i \in \{1, \cdots, k-1\}$, and by Proposition \ref{Prop_Degree lowering} again, we have
    \[
    \left|\Lambda_{\mathcal{P};N}(f_0,...,f_{k})  \right| ^{O_{\mathcal{P}}(1)}\lesssim_{\mathcal{P}}
    \min_{\substack{j=0,\cdots, k \\ j \neq i} }
        \|f_j\|_{{U^{s}}(\TT)}.
    \]
    Therefore, we have the ${U^{s}}(\TT)$ control for all $f_j$, where $j \in \{0, \cdots, k\}$. Repeatedly performing this procedure, we end up with the ${U^{2}}(\TT)$ control
    \begin{equation} \label{proof_U^2}
        \left|\Lambda_{\mathcal{P};N}(f_0,...,f_{k})  \right| ^{O_{\mathcal{P}}(1)}\lesssim_{\mathcal{P}}
    \min_{\substack{j=0,\cdots, k } }
        \|f_j\|_{{U^{2}}(\TT)}.
    \end{equation}
    To go from the ${U^{2}}(\TT)$ control to the ${U^{1}}(\TT)$ control, the claim \eqref{Proof_U^1}, the idea is almost the same as the proof in Proposition \ref{Prop_Intermediate induction step}. Now choose any $i \in \{1, \cdots, k\}$, and recall the dual function
    \[
    F^i(x)= \frac{1}{N} \int_0^N \prod_{\substack{j=0,\cdots, k \\ j \neq i}} f_j(x+P_j(y)-P_i(y)) \,dy.
    \]
    By Lemma \ref{Lma_Dual control}, \eqref{proof_U^2}, and \eqref{U^2-inverse theorem}, we have
    \begin{equation} \label{U^2_Degree lowering_0}
        \left|\Lambda_{\mathcal{P};N}(f_0,...,f_{k})  \right| ^{O_{\mathcal{P}}(1)}\lesssim_{\mathcal{P}} (I)+(II),
    \end{equation}
    where
    \begin{equation} \label{U^2_Degree lowering_I}
        (I):= \left| \frac{1}{N} \int_0^N \int_{\TT} F^i_y(x) \,dx \, dy \right| ,
    \end{equation}
    \begin{equation} \label{U^2_Degree lowering_II}
        (II):=\sup_{\xi_i \in \ZZ-\{0\}} \left| \frac{1}{N} \int_0^N \int_{\TT}e(\xi_i x) F^i_y(x)  \,dx \, dy \right| .
    \end{equation}
    For the first term \eqref{U^2_Degree lowering_I}, by our induction hypothesis that Theorem \ref{Thm_Sobolev} holds for  $\mathcal{P}^i_{k-1}:= \mathcal{P}-\{i\}$, there exist $\sigma''=\sigma''(\mathcal{P})>0$ and $C''=C''(\mathcal{P}), c''=c''(\mathcal{P})>0$ such that
    \[
     \left|\frac{1}{N} \int_0^N \int_{\TT} F^i_y(x) \,dx \, dy  \right|
    =
     \left|\frac{1}{N} \int_0^N \int_{\TT} \prod_{\substack{j=0,\cdots, k \\ j \neq i}} f_j(x+P_j(y)) \,dx \, dy \right|
    \]
    \[
    \leq
        \left|\prod_{\substack{j=0,\cdots, k \\ j \neq i}} \int_{\TT}  f_i \right| + O_{\mathcal{P}} \left(
        N^{-C''} \cdot \min_{\substack{j=0,\cdots, k \\ j \neq i}}  \|f_j\|^{c''}_{\cH^{-\sigma''}(\TT)}
        \right)
    \]
    \[
    \leq
    \min_{\substack{j=0,\cdots, k \\ j \neq i}}  \|f_j\|^{1/2}_{U^1(\TT)} + O_{\mathcal{P}} \left(
        N^{-C''} \cdot \min_{\substack{j=0,\cdots, k \\ j \neq i}}  \|f_j\|^{c''}_{\cH^{-\sigma''}(\TT)}
        \right),
    \]
    where the last inequality follows from the 1-boundedness of $f_j$.\\
    For the second term \eqref{U^2_Degree lowering_II}, by Proposition \ref{Prop_Intermediate induction step} for $(k,l)=(k,1)$, there exist $\sigma'''=\sigma'''(\mathcal{P})>0$ and $C'''=C'''(\mathcal{P}), c'''=c'''(\mathcal{P})>0$ such that
    \[
    \left| \frac{1}{N} \int_0^N \int_{\TT}e(\xi_i x) F^i_y(x)  \,dx \, dy \right|
    \]
    \[
    =\left|
     \frac{1}{N} \int_0^N \int_{\TT}e(\xi_i (x+P_i(y))) \prod_{\substack{j=0,\cdots, k \\ j \neq i}} f_j(x+P_j(y)) \,dx \, dy \right|
    \]
    \[
    \lesssim_{\mathcal{P}}
        N^{-C'''} \cdot \min_{\substack{j=0,\cdots, k \\ j \neq i} }
        \|f_j\|^{c'''}_{\cH^{-\sigma'''}(\TT)}.
    \]
    Finally, combining \eqref{U^2_Degree lowering_0}, and the estimates for \eqref{U^2_Degree lowering_I} and \eqref{U^2_Degree lowering_II} for $i \in \{1,\cdots,k \}$ completes the proof of the claim \eqref{Proof_U^1}; and hence, completes the proof of Theorem \ref{Thm_Sobolev}.
\end{proof}

\section{Proof of Theorem \ref{Thm_Progression}} \label{section 4}
It is known in the literature, as in \cite{FGP22}, that studying the counting operators with Frostman measures as input helps to find nontrivial polynomial progressions, and we largely follow their strategy. However, several technicalities must be overcome.

Firstly, the Sobolev inequality we proved in Theorem \ref{Thm_Sobolev} is not as strong as the one in \cite{FGP22}; hence, we need additional ideas from the property of the Frostman measures to make sense of the counting operators for the Frostman measures.
Secondly, since we aim to find nontrivial polynomial progressions, we must eliminate the contribution of trivial progressions from the counting operators. We achieve this goal by considering the counting operators with a smooth cutoff and studying the corresponding asymptotic formula.
Finally, we will deduce the existence of the nontrivial polynomial progressions from the results obtained in the previous steps.

\subsection{Counting operators for the Frostman measures} \label {Counting operators for the Frostman measures}
This subsection aims to prove the following theorem, which can be regarded as the Sobolev inequality for the Frostman measures. 
\begin{theorem} \label{Thm_Counting}
    Let $k \in \NN$, $N \geq 1$, and let $\mathcal{P}= \{P_1, \cdots, P_{k}\in \RR[y]\}$ be a collection of polynomials with distinct degrees and zero constant terms.
    Then there exist $\epsilon=\epsilon(\mathcal{P})>0$ and $C=C(\mathcal{P})>0$, $c=c(\mathcal{P})>0$ such that
        \[
        \Lambda_{\mathcal{P};N}(\mu_0,...,\mu_k) = \prod_{i=0}^k \int_{\TT} \mu_i
        +O_{\mathcal{P}}\left(N^{-C} \cdot \prod_{i=0}^k  \|\mu_i\|_{\cH^{-\epsilon}(\TT)}^{c} \right) ,
        \]
        for any s-Frostman measures $\mu_i$ with $s \in (1-\epsilon,1]$ and $\text{spt}(\mu_i) \subset \TT$, where $i\in \{0,\cdots,k\}$.
\end{theorem}
As mentioned above, this is a weaker version of the Sobolev inequality compared to the one obtained in \cite{FGP22}; however, it is sufficient for finding polynomial patterns within fractal sets of large dimension. It remains an interesting question to see if one can strengthen Theorem \ref{Thm_Counting} by removing the Frostman measures assumption.

 Next, we will use the following direct consequence of Theorem \ref{Thm_Sobolev}.
\begin{proposition} \label{Prop_Sobolev}
Let $k \in \NN$, $N \geq 1$, and let $\mathcal{P}= \{P_1, \cdots, P_{k}\in \RR[y]\}$ be a collection of polynomials with distinct degrees and zero constant terms. Then there exist $\sigma=\sigma(\mathcal{P})>0$ and $C=C(\mathcal{P})$,  $c=c(\mathcal{P})>0$ \footnote{Without loss of generality, we may assume that $c=c(\mathcal{P}) \in (0,1]$. If $c>1$, we can use the fact that $\|f\|_{\cH^{-\sigma}(\TT)} \leq \|f\|_{\cL^{\infty}(\TT)}$ to reduce the case when $c=1$.} such that
        \[
        \left|\Lambda_{\mathcal{P};N}(f_0,...,f_k) -\prod_{i=0}^k \int_{\TT} f_i \right| \lesssim_{\mathcal{P}}
        N^{-C} \cdot \prod_{i=0}^k \left( \|f_i\|^{1-c}_{\cL^{\infty}(\TT)}  \|f_i\|^{c}_{\cH^{-\sigma}(\TT)} \right) ,
        \]
        for any $f_0, \cdots, f_k \in \cL^{\infty}(\TT)$.
\end{proposition}
The main difference between Proposition \ref{Prop_Sobolev} and Theorem \ref{Thm_Sobolev} is how to remove the $\cL^{\infty}$-control. It is important in practice, since the Frostman measure we use is supported on the fractal set, and it is often an unbounded "function." To overcome this issue, we use the Littlewood-Paley decomposition.

\begin{definition}[Littlewood-Paley decomposition] For any $j \in \NN$, define 
\[
\widehat{\Pi_jf}(\xi):= \widehat{f}(\xi) 1_{\{ |\xi|\in[2^{j-1},2^j) \}}(\xi), \, \text{and} \, \Pi_0f:= \int_{\TT}f,
\]
for any $ f \in \cL^2(\TT)$.
\end{definition}

The next Lemma \ref{Upperbound} plays an essential role in this paper, despite its simplicity. It says that the $\cL^{\infty}$ norm of each Littlewood-Paley piece is well controlled when applied to the Frostman measure. Although the control is not uniformly bounded, we achieve additional savings due to the negative Sobolev norm, which ultimately enables us to prevail.

\begin{lemma} \label{Upperbound}
    Let $\mu$ be any s-Frostman for $s \in (0,1]$, then for any small $\tau>0$, we have
    \[
    \|\Pi_j\mu\|_{\cL^{\infty}(\TT)}
    \lesssim_{\tau} 2^{j(1-s+\tau)}.
    \]
\end{lemma}
\begin{proof} [Proof of Lemma \ref{Upperbound}]
The case for $j=0$ is trivial, so let's assume from now on that $j \in \NN$.
    Recall the Dirichlet kernel
\[
D_M(x):= \sum_{|m| \leq M} e^{2\pi i mx}= \frac{\text{sin}((2M+1)\pi x)}{\text{sin}(\pi x)},
\]
where $M \in \NN$.
We first claim that 
    \begin{equation} \label{Growth rate control}
        \|D_M*\mu\|_{\cL^{\infty}(\TT)} \lesssim M^{1-s} \text{log} M.
    \end{equation}
We can finish the proof quickly with the claim \eqref{Growth rate control}. Note that 
    \[
    \Pi_jf=(D_{2^j-1}-D_{2^{j-1}-1})*f
    \]
And hence,
\[
\|\Pi_j\mu\|_{\cL^{\infty}(\TT)} \leq \|D_{2^j-1}*\mu\|_{\cL^{\infty}(\TT)}+\|D_{2^{j-1}-1}*\mu\|_{\cL^{\infty}(\TT)} 
\]
\[
\lesssim j\cdot 2^{j(1-s)}  \lesssim_{\tau} 2^{j(1-s+\tau)}.
\]
Next, let's return to the proof of the claim \eqref{Growth rate control}. We want to understand the size of $|D_M(x)|$.
We first note that $D_M(x)=D_M(-x)$, and $D_M(x)$ is 1-periodic, so it suffices to to look at $D_M(x)$ in the interval $[0,1/2)$. Also, when $x \in [0,1/2)$,
\[
|\frac{\text{sin}(\pi x)}{\pi x}| \sim 1,
\]
so we may treat 
\[
|D_M(x)| \sim |\frac{\text{sin}((2M+1)\pi x)}{\pi x}|.
\]
Next, we decompose the interval $[0,1/2)$ into a union of smaller consecutive intervals of length $1/(2M+1)$:
\[
[0,1/2)= \left( \bigcup_{j=1}^{M} [\frac{j-1}{(2M+1)},\frac{j}{(2M+1)}) \right)
\bigcup [\frac{M}{(2M+1)},\frac{1}{2}),
\]
where we denote, for $j=1, \cdots, M$,
\[
I_j:= [\frac{j-1}{(2M+1)},\frac{j}{(2M+1)}),
\]
and 
\[
I_{M+1}:= [\frac{M}{(2M+1)},\frac{1}{2}).
\]
Therefore, we have for $j=1, \cdots, M+1$,
\[
\sup_{x\in I_j} |D_M(x)| \lesssim \frac{M}{j}.
\]
Finally, we have
\[
|D_M*\mu(x)| = |\int_0^1 D_M(y) \mu (x-y) \, dy|
\]
\[
\leq |\int_0^{1/2} D_M(y) \mu (x-y) \, dy|+
|\int_{1/2}^1 D_M(y) \mu (x-y) \, dy|.
\]
For the first term, by the definition of the s-Frostman measure,
\[
|\int_0^{1/2} D_M(y) \mu (x-y) \, dy| \leq
\sum_{j=1}^{M+1} \int_{I_j} |D_M(y) \mu (x-y)| \, dy
\]
\[
\lesssim \sum_{j=1}^{M+1} \frac{M}{j} \int_{I_j} \mu (x-y) \, dy
\lesssim \sum_{j=1}^{M+1} \frac{M}{j} \cdot M^{-s} \
\]
\[
\sim M^{1-s} \text{log} M . 
\]
Similarly, for the second term, we also have 
\[
|\int_{1/2}^1 D_M(y) \mu (x-y) \, dy| \lesssim M^{1-s} \text{log} M ,
\]
and this completes the proof.
\end{proof}
Now we are ready to prove Theorem \ref{Thm_Counting}.
\begin{proof} [Proof of Theorem \ref{Thm_Counting}]
    Let $\sigma=\sigma(\mathcal{P})>0$ and $C=C(\mathcal{P})>0$,  $c=c(\mathcal{P})\in (0,1]$ be given in Proposition \ref{Prop_Sobolev}. Let $\epsilon>0$ be chosen later, and $s \in (1-\epsilon,1]$. Note that
    \[
    \Lambda_{\mathcal{P};N}(\Pi_{0}\mu_0,...,\Pi_{0}\mu_k)=\prod_{i=0}^k \int_{\TT} \mu_i,
     \quad \text{and} \quad \int_{\TT} \Pi_j\mu=0, \forall j\geq1.
    \]
Therefore, it suffices to show that
    \[
    \left|\left(\sum_{j_0=0}^\infty \cdots \sum_{j_k=0}^\infty
    \Lambda_{\mathcal{P};N}(\Pi_{j_0}\mu_0,...,\Pi_{j_k}\mu_k)\right)-\Lambda_{\mathcal{P};N}(\Pi_{0}\mu_0,...,\Pi_{0}\mu_k) \right|
    \lesssim_{\mathcal{P}} N^{-C} \cdot \prod_{i=0}^k  \|\mu_i\|_{\cH^{-\epsilon}(\TT)}^{c}.
    \]
    Also note that, for $j \in \NN \bigcup \{0\}$, we have 
    \[
    \|\Pi_jf\|^2_{\cH^{-\sigma}(\TT)}
    \sim 2^{-j \sigma} \cdot \|\Pi_jf\|^2_{\cL^2(\TT)}.
    \]
    If $c=1$, by Proposition \ref{Prop_Sobolev}, we have
    \[
    \left|\left(\sum_{j_0=0}^\infty \cdots \sum_{j_k=0}^\infty
    \Lambda_{\mathcal{P};N}(\Pi_{j_0}\mu_0,...,\Pi_{j_k}\mu_k)\right)-\Lambda_{\mathcal{P};N}(\Pi_{0}\mu_0,...,\Pi_{0}\mu_k) \right| 
    \]
    \[
    \lesssim_{\mathcal{P}}
        N^{-C} \cdot \sum_{j_0=0}^\infty \cdots \sum_{j_k=0}^\infty
        \left(  \prod_{i=0}^k \|\Pi_{j_i}\mu_i\|_{\cH^{-\sigma}(\TT)} \right)
    \]
    \[
    = N^{-C} \cdot \prod_{i=0}^k \left( \sum_{j_i=0}^\infty  \|\Pi_{j_i}\mu_i\|_{\cH^{-\sigma}(\TT)}\right)
    \]
    \[
    \sim N^{-C} \cdot \prod_{i=0}^k \left( \sum_{j_i=0}^\infty  2^{-j_i \sigma/4} \cdot \|\Pi_{j_i}\mu_i\|_{\cH^{-\sigma/2}(\TT)}\right)
    \]
    \[
    \lesssim_{\mathcal{P}}  N^{-C} \cdot \prod_{i=0}^k \|\mu_i\|_{\cH^{-\sigma/2}(\TT)},
    \]
    where we use the Cauchy-Schwarz inequality in the last step. Then, we choose $\epsilon=\frac{\sigma}{2}$ in this case.\\
If $c \in (0,1)$, by Proposition \ref{Prop_Sobolev} and Lemma \ref{Upperbound}, we have
    \[
    \left|\Lambda_{\mathcal{P};N}(\Pi_{j_0}\mu_0,...,\Pi_{j_k}\mu_k) \right|
    \]
    \[
    \lesssim_{\mathcal{P}}
        N^{-C} \cdot \prod_{i=0}^k \left( \|\Pi_{j_i}\mu_i\|^{1-c}_{\cL^{\infty}(\TT)}  \|\Pi_{j_i}\mu_i\|^{c}_{\cH^{-\sigma}(\TT)} \right) 
    \]
    \[
    \lesssim_{\tau} 
    N^{-C} \cdot 2^{[(1-c)(1-s+\tau)-c\sigma/2](\sum_{i=0}^k j_i)} \cdot \prod_{i=0}^k  \|\Pi_{j_i}\mu_i\|^c_{\cL^2(\TT)}  ,
    \]
    for any small $\tau>0$, and any $(j_0,\cdots,j_k) \in \left( \NN \cup \{0\} \right)^{k+1}-\{(0,\cdots,0)\}$. Let $\epsilon<\frac{c\sigma}{9(1-c)}$, $s \in (1-\epsilon,1]$, and $\tau= \frac{c\sigma}{8(1-c)}$, so that 
    \[
    (1-c)(1-s+\tau)-\frac{c\sigma}{2}< -\frac{c\sigma}{4}
    \Leftrightarrow 1-s< \frac{c\sigma}{8(1-c)}.
    \]
    To conclude, by our choice of the parameters, we have shown that
    \[
    \left|\Lambda_{\mathcal{P};N}(\Pi_{j_0}\mu_0,...,\Pi_{j_k}\mu_k) \right| \lesssim_{\mathcal{P}}
    N^{-C} \cdot 2^{(-c\sigma/4)(\sum_{i=0}^k j_i)}\cdot \prod_{i=0}^k  \|\Pi_{j_i}\mu_i\|^c_{\cL^2(\TT)}.
    \]
    Therefore, we have
    \[
    \left|\left(\sum_{j_0=0}^\infty \cdots \sum_{j_k=0}^\infty
    \Lambda_{\mathcal{P};N}(\Pi_{j_0}\mu_0,...,\Pi_{j_k}\mu_k)\right)-\Lambda_{\mathcal{P};N}(\Pi_{0}\mu_0,...,\Pi_{0}\mu_k) \right| 
    \]
    \[
    \lesssim_{\mathcal{P}} N^{-C} \cdot
    \sum_{j_0=0}^\infty \cdots \sum_{j_k=0}^\infty  2^{(-c\sigma/4)(\sum_{i=0}^k j_i)}\cdot \prod_{i=0}^k \ \|\Pi_{j_i}\mu_i\|^c_{\cL^2(\TT)} 
    \]
    \[
    = N^{-C} \cdot \prod_{i=0}^k \left( \sum_{j_i=0}^\infty  2^{-j_ic \sigma/4} \cdot \|\Pi_{j_i}\mu_i\|^c_{\cL^{2}(\TT)}\right)
    \]
    \[
    \sim N^{-C} \cdot \prod_{i=0}^k \left( \sum_{j_i=0}^\infty  2^{-j_ic \sigma/8} \cdot \|\Pi_{j_i}\mu_i\|^c_{\cH^{-\sigma/4}(\TT)}\right)
    \]
    \[
    \lesssim_{\mathcal{P}} N^{-C} \cdot \prod_{i=0}^k  \|\mu_i\|_{\cH^{-\sigma/4}(\TT)}^{c},
    \]
     where we use Hölder's inequality in the last step. Finally, we choose $\epsilon=\min\{\frac{c\sigma}{9(1-c)},\frac{\sigma}{4}\}$ in this case, and the proof is complete.
\end{proof}

\subsection{Counting operators with a smooth cutoff}
In this subsection, we want to obtain a similar consequence as in Theorem \ref{Thm_Counting} when the $y$-variable is contained in an interval away from the origin. The main reason is that we want to find a nontrivial polynomial progression corresponding to the condition that $y \neq 0$.
\\Let $\chi$ be any smooth cutoff function supported on $[0,1]$, and $\chi_N(y):=\frac{1}{N} \chi(\frac{y}{N})$. We define the counting operators with a smooth cutoff as follows
\[
\Lambda_{\mathcal{P};N}(\chi; f_0,...,f_{k}):=
    \frac{1}{N} \int_0^N \int_{\TT}f_0(x) \prod_{i=1}^k f_i( x+ P_i(y)) \chi_N(y) \, dx \,dy.
\]

Now we state a version of the Sobolev inequality with a smooth cutoff.
\begin{proposition} \label{Prop_Sobolev_cutoff}
Let $k \in \NN$, $N \geq 1$, and $f_0, \cdots, f_k$ be 1-bounded. Let $\mathcal{P}= \{P_1, \cdots, P_{k}\in \RR[y]\}$ be a collection of polynomials with distinct degrees and zero constant terms. Then there exist $\sigma=\sigma(\mathcal{P})>0$ and $C=C(\mathcal{P})$,  $c=c(\mathcal{P})>0$ such that
        \[
        \left|\Lambda_{\mathcal{P};N}(\chi;f_0,...,f_k) - \int_0^1\chi \cdot \prod_{i=0}^k \int_{\TT} f_i \right| \lesssim_{\mathcal{P}, \chi}
        N^{-C} \cdot \min_{i=0,\cdots, k}    \|f_i\|^{c}_{\cH^{-\sigma}(\TT)}.
        \]
\end{proposition}
\begin{proof}[Proof of Proposition \ref{Prop_Sobolev_cutoff}]
    The proof is quite similar to the proof of Theorem \ref{Thm_Sobolev}. However, we can apply Theorem \ref{Thm_Sobolev} as a black box instead of repeating all the arguments. Therefore, we will only sketch the proof idea, and refer the readers to the details of those missing steps in the proof of Theorem \ref{Thm_Sobolev}.\\ 
    The key idea is that we can obtain a $U^2$-norm control without using PET induction. A similar idea has appeared in the proof of Lemma \ref{Lma_Intermediate induction step}. We claim that
    \begin{equation} \label{special U^2 control with smooth cutoff}
        \left|\Lambda_{\mathcal{P};N}(\chi;f_0,...,f_{k}) \right|^{O_{\mathcal{P}}(1)} \lesssim_{\mathcal{P},\chi}
    \min_{\substack{j=0,\cdots, k} }\|f_j\|_{U^2(\TT)}.
    \end{equation}
    The proof is almost identical to that of \eqref{Special_U^2 control}. The only difference is that we apply Theorem \ref{Thm_Sobolev} instead of the induction hypothesis used in the proof of \eqref{Special_U^2 control}. Next, we define the i-th dual function with a smooth cutoff as follows:
\[
    F^i_{\chi}(x)=F^i_{\mathcal{P};N;\chi}(x):= \frac{1}{N} \int_0^N \prod_{\substack{j=0,\cdots, k \\ j \neq i}} f_j(x+P_j(y)-P_i(y)) \chi_N(y)\,dy,
\]
where we denote $P_0(y)=0$.
Then we have 
\begin{equation} \label{pairing}
    \Lambda_{\mathcal{P};N}(\chi;f_0,...,f_k)= \int_{\TT} f_i(x)F^i_{\chi}(x) \, dx.
\end{equation}
Apply the Cauchy-Schwarz inequality to \eqref{pairing} as in the Lemma \ref{Lma_Dual control}, and use \eqref{special U^2 control with smooth cutoff}, we conclude that
\begin{equation} \label{special U^2 control with smooth cutoff for dual}
        \left|\Lambda_{\mathcal{P};N}(\chi;f_0,...,f_{k}) \right|^{O_{\mathcal{P}}(1)} \lesssim_{\mathcal{P},\chi}
    \min_{\substack{j=0,\cdots, k} }\|F^j_{\chi}\|_{U^2(\TT)}.
    \end{equation}
    Finally, we can repeat the same argument as in Section \ref{Proof of Theorem Thm_Sobolev} to deduce the conclusion of Proposition \ref{Prop_Sobolev_cutoff} from \eqref{special U^2 control with smooth cutoff for dual}. The only additional ingredient is the following variant of the oscillatory integral estimate (see \cite[Proposition 2.1]{SW01}, for example)
    \begin{equation} \label{Oscillatory integral_ cutoff}
        \left|\int_{0}^1  e(P(y)) \chi(y)\, dy\right| \lesssim_{d,\chi} (\max_i \, |a_i|)^{-1/d},
    \end{equation}
    where $P(y)= \sum_{i=1}^d a_i y^i \in \RR[y]$. We replace the use of Lemma \ref{Lma_Oscillatory integral} by \eqref{Oscillatory integral_ cutoff} in Section \ref{Proof of Theorem Thm_Sobolev} (also, the corresponding intermediate induction step in Section \ref{Intermediate induction step}) in the last step, and the proof is complete.
\end{proof}
Finally, we can follow the same argument as in Section \ref{Counting operators for the Frostman measures} to upgrade from Proposition \ref{Prop_Sobolev_cutoff} to Proposition \ref{prop_counting_cutoff}, which establishes the Sobolev inequality for the Frostman measures with a smooth cutoff. 

\begin{proposition} \label{prop_counting_cutoff}
     Let $k \in \NN$, $N \geq 1$, and let $\mathcal{P}= \{P_1, \cdots, P_{k}\in \RR[y]\}$ be a collection of polynomials with distinct degrees and zero constant terms.
    Then there exist $\epsilon=\epsilon(\mathcal{P})>0$ and $C=C(\mathcal{P})>0$, $c=c(\mathcal{P})>0$ such that
        \[
        \Lambda_{\mathcal{P};N}(\chi;\mu_0,...,\mu_k) = \int_0^1\chi \cdot\prod_{i=0}^k \int_{\TT} \mu_i
        +O_{\mathcal{P},\chi}\left(N^{-C} \cdot \prod_{i=0}^k  \|\mu_i\|_{\cH^{-\epsilon}(\TT)}^{c} \right) ,
        \]
        for any s-Frostman measures $\mu_i$ with $s \in (1-\epsilon,1]$ and $\text{spt}(\mu_i) \subset \TT$, where $i\in \{0,\cdots,k\}$.
\end{proposition}
\begin{remark} \label{rmk_counting_cutoff}
    The conclusion of \ref{prop_counting_cutoff} can apply to a broader class of measures than Frostman measures. If we look back at the proof, we only use the Frostman condition when applying Lemma \ref{Upperbound}. Therefore, the conclusion still holds if we allow the measure $\mu$ to be a complex measure, and its absolute value $|\mu|$ still satisfies the Frostman condition.
\end{remark}

\subsection{Finding the nontrivial progressions}
This subsection aims to prove Theorem \ref{Thm_Progression}, following \cite[Section 6]{FGP22} for the most part.
Recall the Fejér kernel
\[
K_M(x):= \sum_{|m| \leq M} \left(  1-\frac{|m|}{M+1} \right) e^{2\pi i mx}= \frac{1}{M+1} \left( \frac{\text{sin}((2M+1)\pi x)}{\text{sin}(\pi x)} \right)^2,
\]
where $M \in \NN$. Also, note that $K_M \geq 0$ and $\int_{\TT} K_M(x) \, dx =1$.\\
We start with constructing a nonnegative measure $\nu$ supported on the set of polynomial progressions.
\begin{proposition} \label{measure for the progression}
    Let $k \in \NN$, $N \geq 1$, and let $\mathcal{P}= \{P_1, \cdots, P_{k}\in \RR[y]\}$ be a collection of polynomials with distinct degrees and zero constant terms.
    Then there exist $\epsilon=\epsilon(\mathcal{P})>0$ such that for any s-Frostman measures $\mu$ with $s \in (1-\epsilon,1]$ and $\text{spt}(\mu) \subset \TT$, the following linear functional $\nu(\mu)=\nu_{\mathcal{P};N;\chi}(\mu)$ is well-defined
    \begin{equation} \label{nu}
        \langle \nu(\mu),g \rangle :=
        \lim_{M \to \infty} \frac{1}{N}\int_{0}^N \int_{\TT} g(x,y) (K_M*\mu)(x) \prod_{i=1}^k (K_M*\mu)( x+ P_i(y)) \chi_N(y) \, dx \,dy,
    \end{equation}
    for every continuous function $g:\TT \times [0,N] \to \RR$. Moreover, we have 
    \begin{equation} \label{nu_infty}
        |\langle \nu(\mu),g \rangle| \lesssim_{\mathcal{P},N,\chi,\mu}
    \|g\|_{\cL^{\infty}(\TT \times [0,N])}.
    \end{equation}
\end{proposition}

\begin{proof}[Proof of Proposition \ref{measure for the progression}]
    Let $\epsilon$ be given as in Proposition \ref{prop_counting_cutoff}, and $\mu$ be any s-Frostman measure with $s \in (1-\epsilon,1]$. Note that for any $M \in \NN$, $K_M*\mu$ is also an s-Frostman measure. Moreover, we have $\int_{\TT}K_M*\mu=\int_{\TT}\mu$, and $\|K_M*\mu\|_{\cH^{-\epsilon}(\TT)}\leq \|\mu\|_{\cH^{-\epsilon}(\TT)}$. Therefore, by Proposition \ref{prop_counting_cutoff}, we have
    \[
    \left| \frac{1}{N}\int_{0}^N \int_{\TT} g(x,y) (K_M*\mu)(x) \prod_{i=1}^k (K_M*\mu)( x+ P_i(y)) \chi_N(y) \, dx \,dy \right|
    \]
    \[
    \leq \|g\|_{\cL^{\infty}(\TT \times [0,N])} \cdot
    \left( \frac{1}{N}\int_{0}^N \int_{\TT}  (K_M*\mu)(x) \prod_{i=1}^k (K_M*\mu)( x+ P_i(y)) \chi_N(y) \, dx \,dy \right)
    \]
    \[
    \leq \|g\|_{\cL^{\infty}(\TT \times [0,N])} \cdot
    \left( \int_0^1\chi \cdot (\int_{\TT} \mu)^{k+1}
        +O_{\mathcal{P},\chi}\left(N^{-C} \cdot   \|\mu\|_{\cH^{-\epsilon}(\TT)}^{(k+1)c} \right) \right)
    \]
    \[
    \lesssim_{\mathcal{P},N,\chi,\mu}
    \|g\|_{\cL^{\infty}(\TT \times [0,N])},
    \]
    for any $M \in \NN$, and this proves \eqref{nu_infty}.\\
    Next, to show that \eqref{nu} is well-defined for continuous functions $g:\TT \times [0,N] \to \RR$, it suffices to consider the following special form by density argument:
    \[
    g(x,y)=e(l_1x+\frac{l_2}{N}y)
    \]
    where $l_1,l_2 \in \NN$. To be more precise, we approximate any continuous function by a function whose Fourier series consists of a finite number of terms. Without loss of generality, we assume that $1 \leq \text{deg}(P_1) < \cdots < \text{deg}(P_{k})$. If $\text{deg}(P_1)=1$, then $P_1(y)=cy$ for some $c\in\RR -\{0\}$. In this case, we can write
    \[
    e(l_1x+\frac{l_2}{N}y) f(x) f(x+cy)=
    f'(x) f'' (x+cy),
    \]
    where $f'(x):=e((l_1-\frac{l_2}{cN})x)f(x)$ and $f''(x):=e(\frac{l_2}{cN}x)f(x)$. Note that we have
    \[
    K_M*\mu \rightarrow \mu, \, (K_M*\mu)' \rightarrow \mu', \, (K_M*\mu)'' \rightarrow \mu''\quad \text{in}\quad \cH^{-\eps}(\TT).
    \]
    Therefore, \eqref{nu} is well-defined by applying Proposition \ref{prop_counting_cutoff} and Remark \ref{rmk_counting_cutoff} if $\text{deg}(P_1)=1$.\\
    If $\text{deg}(P_1)>1$, we set $P_0(y)=\frac{l_2}{N}y$, if $l_2 \neq 0$, and consider $\mathcal{P}':=\mathcal{P}\cup\{P_0\}$. In this case, we can write
    \[
    e(l_1x+\frac{l_2}{N}y) f(x)=
    f'(x) e(x+\frac{l_2}{N}y),
    \]
    where $f'(x):=e((l_1-1)x)f(x)$. Therefore, \eqref{nu} is well-defined by applying Proposition \ref{prop_counting_cutoff} and Remark \ref{rmk_counting_cutoff} to $\mathcal{P}'$ with the same argument we have explained above.
    Finally, if $l_2=0$, we simply write $f'(x):=e(l_1x)f(x)$, and we can use the same argument to deduce the well-definedness of \eqref{nu}. Therefore, the proof is complete.
\end{proof}

Now, we are ready to prove Theorem \ref{Thm_Progression}.
\begin{proof}[Proof of Theorem \ref{Thm_Progression}]
    Let $\epsilon$ be given as in Proposition \ref{prop_counting_cutoff}, and $E \subset \TT$ with $\dH (E)>1-\epsilon$. By Frostman Lemma, there exists an s-Frostman measure $\mu_E$ with $\text{spt}(\mu_E) \subset E$, where $s=\dH (E)$.
Next, we choose any smooth, nonnegative cutoff function $\chi$ with $\text{spt}(\chi) \subset [1/10,1]$, and $\int_0^1\chi>0$. By Proposition \ref{prop_counting_cutoff}, we can choose $N$ large enough so that 
    \begin{equation} \label{lower bound of measure}
        \Lambda_{\mathcal{P};N}(\chi;\mu_E,...,\mu_E) >\frac{1}{2}\left( \int_0^1\chi \cdot (\int_{\TT} \mu_E)^{k+1} \right)>0.
    \end{equation}
    With the choice of $\mathcal{P},N,\chi$ and $\mu_E$, we consider the linear functional $\nu(\mu_E)=\nu_{\mathcal{P};N;\chi}(\mu_E)$ as defined in \eqref{nu}. Then by Proposition \ref{measure for the progression} and the Riesz representation theorem, $\nu(\mu_E)$ is a nonnegative measure. Finally, it is a routine check that 
    \begin{equation} \label{support of nu}
        \text{spt}(\nu(\mu_E)) \subset \{(x,y)\in\TT\times[\frac{N}{10},N]: x,x+P_1(y),\cdots,x+P_k(y) \in E\}.
    \end{equation}
    Therefore, by \eqref{lower bound of measure} and \eqref{support of nu},  we can find $y\neq 0$ so that $\{x,x+P_1(y), \cdots,x+P_k(y)\} \subset E$, and we are done.
\end{proof}

\section{Proof of Theorem \ref{Thm_Convergence}}
To prove the almost everywhere convergence result, a strong quantitative estimate for the norm convergence, together with a lacunary subsequence trick, is sufficient. A similar strategy has appeared before, as in \cite{CDKR22}. For a more detailed discussion of the lacunary subsequence tricks, we refer readers to \cite[Appendix]{FLW12}, for example.

Before we start the proof, note that for any $f \in \cL^{\infty}(\TT)$ with $f \geq0$, and $f \neq 0$, the measure defined by $\frac{f(x)}{\|f\|_{\cL^{\infty}}} \,dx$ is a 1-Frostman measure.
    We will use the following direct consequence of Theorem \ref{Thm_Counting}.\footnote{For any general 1-bounded function $f$, we may decompose $f$ as $f^+ + f^-$, where $f^+(x) := f(x)1_{\{f(x)\geq0\}}(x)$ and $f^-(x) := f(x)1_{\{f(x)<0\}}(x)$. And then we apply Theorem \ref{Thm_Counting} to $f^+$ and $f^-$ respectively.}
    \begin{proposition} \label{Prop_Counting}
        Let $k \in \NN$, $N \geq1$, and let $f_0, \cdots, f_{k}$ be 1-bounded. Let $\mathcal{P}= \{P_1, \cdots, P_{k}\in \RR[y]\}$ be a collection of polynomials with distinct degrees and zero constant terms.
    Then there exist $\epsilon=\epsilon(\mathcal{P})>0$ and $C=C(\mathcal{P})>0$ such that
        \[
        \left| \Lambda_{\mathcal{P};N}(\mu,f_1,...,f_k) -\int_{\TT} \mu \cdot \prod_{i=1}^k \int_{\TT} f_i \right|
        \lesssim_{\mathcal{P}}N^{-C} ,
        \]
        for any s-Frostman measures $\mu$ with $s \in (1-\epsilon,1]$ and $\text{spt}(\mu) \subset \TT$.
    \end{proposition}
Now, we first state our strong quantitative estimates for the norm convergence.
\begin{proposition} \label{Prop_norm convergence}
    Let $k \in \NN$, $N \geq1$, and let $f_0, \cdots, f_{k}$ be 1-bounded. Let $\mathcal{P}= \{P_1, \cdots, P_{k}\in \RR[y]\}$ be a collection of polynomials with distinct degrees and zero constant terms.
    Then there exist $\epsilon=\epsilon(\mathcal{P})>0$ and $C=C(\mathcal{P})>0$ such that
    \[
    \|\mathcal{A}_{\mathcal{P};N}(f_1,...,f_k)-\prod_{i=1}^k \int_{\TT}f_i\|_{\cL^1(\mu)} \lesssim_{\mathcal{P}} N^{-C},
    \]
    for any s-Frostman measures $\mu$ with $s \in (1-\epsilon,1]$ and $\text{spt}(\mu) \subset \TT$.
\end{proposition}

\begin{proof} [Proof of Proposition \ref{Prop_norm convergence}]
Let $\epsilon>0$ and $C>0$ be given in Proposition \ref{Prop_Counting}.
    Note that 
    \[
    \|\mathcal{A}_{\mathcal{P};N}(f_1,...,f_k)\|_{\cL^1(\mu)}
    = \int_{\TT} \mu(x) |\mathcal{A}_{\mathcal{P};N}(f_1,...,f_k)(x)| \, dx
    \]
    \[
    = \int_{\TT} \mu(x) S_{\mathcal{P};N}(f_1,...,f_k)(x) \mathcal{A}_{\mathcal{P};N}(f_1,...,f_k)(x) \, dx,
    \]
    where $S_{\mathcal{P};N}(f_1,...,f_k)(x)$ is defined so that 
    \[
    |\mathcal{A}_{\mathcal{P};N}(f_1,...,f_k)(x)|=S_{\mathcal{P};N}(f_1,...,f_k)(x)\mathcal{A}_{\mathcal{P};N}(f_1,...,f_k)(x).
    \]
    If $\mathcal{A}_{\mathcal{P};N}(f_1,...,f_k)(x)=0$, we define that $S_{\mathcal{P};N}(f_1,...,f_k)(x)=1$.
    From the definition, we have $|S_{\mathcal{P};N}(f_1,...,f_k)(x)|=1$. According to the sign of $S_{\mathcal{P};N}(f_1,...,f_k)(x)$, we define
    \[
    S^{+}_{\mathcal{P};N}(f_1,...,f_k)(x) := S_{\mathcal{P};N}(f_1,...,f_k)(x) 1_{\{S_{\mathcal{P};N}(f_1,...,f_k)(x)>0\}}(x),
    \]
    \[
    S^{-}_{\mathcal{P};N}(f_1,...,f_k)(x) := S_{\mathcal{P};N}(f_1,...,f_k)(x) 1_{\{S_{\mathcal{P};N}(f_1,...,f_k)(x)<0\}}(x),
    \]
    so that we have
    \[
    S_{\mathcal{P};N}(f_1,...,f_k)(x)=
    S^{+}_{\mathcal{P};N}(f_1,...,f_k)(x)+S^{-}_{\mathcal{P};N}(f_1,...,f_k)(x).
    \]
    We first assume that 
    \[
    \prod_{i=1}^k \int_{\TT} f_i=0.
    \]
    By using the easy fact that both $S^{+}_{\mathcal{P};N}\mu$ and $-S^{-}_{\mathcal{P};N}\mu$ are s-Frostman measures, by Proposition \ref{Prop_Counting}, we have
    \[
    \|\mathcal{A}_{\mathcal{P};N}(f_1,...,f_k)\|_{\cL^1(\mu)}
    \]
    \[
    =\Lambda_{\mathcal{P};N}(S^{+}_{\mathcal{P};N}\mu,f_1,...,f_k)+\Lambda_{\mathcal{P};N}(S^{-}_{\mathcal{P};N}\mu,f_1,...,f_k)
    \]
    \[
    \lesssim_{\mathcal{P}} N^{-C}.
    \]
     Finally, if we denote$\int_{\TT}f_i:=f_{i,0}$ and $f_i-\int_{\TT}f_i:=f_{i,1}$, then we have $f_i=f_{i,0}+f_{i,1}$, and $\int_{\TT}f_{i,1}=0$. Therefore,
    \[
    \|\mathcal{A}_{\mathcal{P};N}(f_1,...,f_k)-\prod_{i=1}^k \int_{\TT}f_i\|_{\cL^1(\mu)}
    \]
    \[
    =\|\sum_{\substack{\underline{w}\in\{0,1\}^k\\ \exists i, w_i\neq0 }}\mathcal{A}_{\mathcal{P};N}(f_{1,w_1},...,f_{k,w_k})\|_{\cL^1(\mu)}
    \]
    \[
    \leq \sum_{\substack{\underline{w}\in\{0,1\}^k\\ \exists i, w_i=1 }} \|\mathcal{A}_{\mathcal{P};N}(f_{1,w_1},...,f_{k,w_k})\|_{\cL^1(\mu)}
    \]
    \[
    \lesssim_{\mathcal{P}} N^{-C},
    \]
    and this completes the proof.
\end{proof}
The next step is to upgrade Proposition \ref{Prop_norm convergence}, the norm convergence result, to the almost everywhere convergence result by applying the lacunary subsequence trick.
\begin{proposition}\label{Prop_AEConvergence}
    Let $k \in \NN$ and $\mathcal{P}= \{P_1, \cdots, P_{k}\in \RR[y]\}$ be a collection of polynomials with distinct degrees and zero constant terms. Then there exists $\epsilon=\epsilon(\mathcal{P})>0$ such that given any $f_1, \cdots, f_k \in \cL^{\infty}(\TT)$, and s-Frostman measures $\mu$ with $s \in (1-\epsilon,1]$ and $\text{spt}(\mu) \subset \TT$,
    \[
    \lim_{N \to \infty} \mathcal{A}_{\mathcal{P};N}(f_1,...,f_k)(x)=
    \prod_{i=1}^k \int_{\TT}f_i
    \]
     holds for $\mu$-almost every $x$.
\end{proposition}

At first glance, the pointwise convergence result in Proposition \ref{Prop_AEConvergence} only holds for $\mu$-almost every $x$. However, since the conclusion is true for any s-Frostman measures $\mu$ with $s \in (1-\epsilon,1]$,
we can upgrade it to the $\cH^s$-almost everywhere convergence result by a standard reduction. Let's first see this implication.
\begin{proof} [Proof of Theorem \ref{Thm_Convergence} assuming Proposition \ref{Prop_AEConvergence}]
Given any $s \in (1-\epsilon,1]$ from Proposition \ref{Prop_AEConvergence}. Consider the following divergent set:
\[
\mathcal{D}:= \left\{ x \in \TT: \lim_{N \to\infty}\mathcal{A}_{\mathcal{P};N}(f_1,...,f_k)(x) \neq \prod_{i=1}^k \int_{\TT} f_i \right\} .
\]
It is not hard to see that $\mathcal{D}$ is a Borel set.
Suppose that $\dH(\mathcal{D})\geq s$, then by Frostman lemma, there exist an s-Frostman $\mu$ with $\text{spt}(\mu) \subset \mathcal{D}$. By Proposition \ref{Prop_AEConvergence}, we have for $\mu$-almost every $x$, we have
    \[
    \lim_{N \to \infty} \mathcal{A}_{\mathcal{P};N}(f_1,...,f_k)(x)=
    \prod_{i=1}^k \int_{\TT}f_i.
    \]
However, this contradicts our definition of the divergence set $\mathcal{D}$. Therefore, $\dH(\mathcal{D})< s$ for all $s \in (1-\epsilon,1]$. In particular, $\dH(\mathcal{D}) \leq 1-\epsilon$, and this finishes the proof.
\end{proof}

Finally, it remains to prove Proposition \ref{Prop_AEConvergence}.
\begin{proof} [Proof of Proposition \ref{Prop_AEConvergence}]
    Let $\epsilon>0$ be given so that Proposition \ref{Prop_norm convergence} is true. After normalization, we may assume that $f_1, \cdots, f_k$ are 1-bounded. Consider the following lacunary sequence $\{N(\tau,l):=(1+\tau)^l\}_{l \in \NN}$. We first claim that
    \begin{equation} \label{Application of B-C Lemma(III)}
        \mu\left( \bigcup_{m\in \NN} \left\{ x \in \TT: \lim_{l \to\infty}\mathcal{A}_{\mathcal{P};N(\frac{1}{m},l)}(f_1,...,f_k)(x) \neq \prod_{i=1}^k \int_{\TT} f_i \right\}\right)=0.
    \end{equation}
    Given any $\delta>0$, denote 
    \[
    \mathcal{D}(\tau,l,\delta):= \left\{ x \in \TT: \left|\mathcal{A}_{\mathcal{P};N(\tau,l)}(f_1,...,f_k)(x) - \prod_{i=1}^k \int_{\TT} f_i \right| > \delta \right\}.
    \]
By Proposition \ref{Prop_norm convergence}, there exists $C=C(\mathcal{P})>0$ such that
\[
\sum_{l=1}^{\infty} \mu \left(  \mathcal{D}(\tau,l,\delta) \right) \lesssim_{\mathcal{P}} \frac{1}{\delta} \cdot \sum_{l=1}^{\infty}  N_l^{-C} 
\]
\[
= \frac{1}{\delta} \cdot \sum_{l=1}^{\infty}  (1+\tau)^{-Cl}  < \infty,
\]
for any $\tau>0$. Therefore, by the Borel–Cantelli lemma, we have
\begin{equation} \label{Application of B-C Lemma(II)}
    \mu \left( \limsup_{l \to\infty} \mathcal{D}(\tau,l,\delta) \right)=0.
\end{equation}
Note that 
\[
\left\{ x \in \TT: \lim_{l \to\infty}\mathcal{A}_{\mathcal{P};N(\tau,l)}(f_1,...,f_k)(x) \neq \prod_{i=1}^k \int_{\TT} f_i \right\} \subset
\bigcup_{m \in \NN} \limsup_{l \to\infty} \mathcal{D}(\tau,l,\frac{1}{m}).
\]
Applying the union bound on \eqref{Application of B-C Lemma(II)} implies that for any $\tau>0$,
    \begin{equation} \label{Application of B-C Lemma}
        \mu\left(  \left\{ x \in \TT: \lim_{l \to\infty}\mathcal{A}_{\mathcal{P};N(\tau,l)}(f_1,...,f_k)(x) \neq \prod_{i=1}^k \int_{\TT} f_i \right\}  \right)=0.
    \end{equation} \\
Applying the union bound again on \eqref{Application of B-C Lemma} gives the proof of \eqref{Application of B-C Lemma(III)}.\\
Next, we claim that
    \begin{equation} \label{difference}
        \sup_{M \in [N(\tau,l),N(\tau,l+1)]}
    |\mathcal{A}_{\mathcal{P};N(\tau,l)}(x)-\mathcal{A}_{\mathcal{P};M}(x)|\lesssim \tau.
    \end{equation}
    for any $l \in \NN$. 
Note that for any $M \in [N(\tau,l),N(\tau,l+1)]$, we have
\[
|\mathcal{A}_{\mathcal{P};N(\tau,l)}(x)-\mathcal{A}_{\mathcal{P};M}(x)|
\]
\[
\leq  \left(\frac{M-N(\tau,l)}{M} \right) \cdot \left| \frac{1}{N(\tau,l)} \int_0^{N(\tau,l)} \prod_{\substack{j=1,\cdots, k}} f_j(x+P_j(y)) \,dy \right| 
\]
\[
+ \left( \frac{M-N(\tau,l)}{M} \right) \cdot \left| \frac{1}{M-N(\tau,l)} \int_{N(\tau,l)}^M \prod_{\substack{j=1,\cdots, k}} f_j(x+P_j(y)) \,dy \right| 
\]
\[
\lesssim \left( \frac{M-N(\tau,l)}{M} \right) 
\lesssim \left( \frac{N(\tau,l+1)-N(\tau,l)}{N(\tau,l)} \right) 
=\tau,
\]
as desired. Finally, combining \eqref{Application of B-C Lemma(III)} and \eqref{difference} shows that for $\mu$-almost every $x$, we have
    \[
    \lim_{N \to \infty} \mathcal{A}_{\mathcal{P};N}(f_1,...,f_k)(x)=
    \prod_{i=1}^k \int_{\TT}f_i,
    \]
    which finishes the proof.
\end{proof}

\bibliographystyle{alpha}
\bibliography{bibliography}

\end{document}